\documentclass[12pt, reqno]{amsart}
\usepackage{amsmath, amsthm, amscd, amsfonts, amssymb, graphicx, color, mathrsfs}
\usepackage[bookmarksnumbered, colorlinks, plainpages]{hyperref}
\usepackage[all]{xy}
\usepackage{slashed}

\usepackage{soul}
\usepackage{cancel}
\usepackage{ulem}

\newtheorem{theorem}{Theorem}[section]

\newtheorem{proposition}[theorem]{Proposition}

\theoremstyle{definition}

\newtheorem{example}[theorem]{Example}

\theoremstyle{remark}
\newtheorem{remark}[theorem]{Remark}
\numberwithin{equation}{section}

\begin{document}
\setcounter{page}{1}

\title[  On the  nuclear trace  of Fourier integral  operators in $L^p$-spaces ]{ On the  nuclear trace  of Fourier integral  operators in $L^p$-spaces}

\author[D. Cardona]{Duv\'an Cardona}
\address{
  Duv\'an Cardona:
  \endgraf
  Department of Mathematics  
  \endgraf
  Pontificia Universidad Javeriana
  \endgraf
  Bogot\'a
  \endgraf
  Colombia
  \endgraf
  {\it E-mail address} {\rm d.cardona@uniandes.edu.co;
duvanc306@gmail.com}
  }



\begin{abstract}
In this paper we provide characterizations for the nuclearity of Fourier integral operators on $\mathbb{R}^n,$ on the discrete group $\mathbb{Z}^n,$ arbitrary compact Lie groups and compact homogeneous manifolds.  We also investigate the nuclear trace of these operators.
\end{abstract} \maketitle
\tableofcontents

\section{Introduction}
In this paper we characterize the $r$-nuclearity of Fourier integral operators on Lebesgue spaces. Fourier integral operators will be considered in $\mathbb{R}^n,$ the discrete group $\mathbb{Z}^n,$ the $n$-dimensional torus, arbitrary compact Lie groups and symmetric spaces (compact homogeneous manifolds). We also  give formulae for the nuclear trace of these operators. Explicit examples will be given on $\mathbb{Z}^n,$ the torus $\mathbb{T}^n$, the special unitary group $\textnormal{SU}(2),$ and the projective complex plane $\mathbb{C}\mathbb{P}^2.$ Our main theorems will be applied to the characterization of $r$-nuclear pseudo-differential operators defined by the Weyl quantization procedure.

\subsection{Outline of the paper}

Let us recall that  the Fourier integral operators (FIOs) on $\mathbb{R}^n,$ are integral operators  of the form
\begin{equation}\label{definition}
Ff(x):=\int_{\mathbb{R}^n}e^{i\phi(x,\xi)}a(x,\xi)(\mathscr{F}f)(\xi)d\xi,\,\,\,\,\,\,\,\,\,\,\,\,\,\,\,\,\,\,
\end{equation}
where $\mathscr{F}_{ }f$ is the Fourier transform of  $f,$ or in a more general setting, linear integral operators  formally defined by
\begin{equation}
Tf(x):=\int_{\mathbb{R}^{2n}}e^{i\phi(x,\xi)-i2\pi  y\cdot\xi}a(x,y,\xi)f(y)dyd\xi.
\end{equation}
As it is well known, FIOs are used to 
express solutions to Cauchy problems of hyperbolic equations  as well as  for  obtaining   asymptotic formulas for the Weyl eigenvalue function associated to geometric operators (see H\"ormander \cite{Hor71,Hor1,Hor2} and Duistermaat and H\"ormander \cite{DuiHor}).  

According to the theory of FIOs developed by H\"ormander \cite{Hor71}, the phase functions $\phi$  are positively homogeneous
of order 1 and they are considered smooth at $\xi\neq 0,$  while the  symbols are considered satisfying estimates of the form
\begin{equation}
\sup_{(x,y)\in K}|\partial_x^\beta\partial^\alpha_\xi a(x,y,\xi)|\leq C_{\alpha,\beta,K}(1+|\xi|)^{\kappa-|\alpha|},
\end{equation}
for every compact subset $K$ of $\mathbb{R}^{2n}.$  Let us observe that   $L^p$-properties for FIOs can be found in the references
H\"ormander \cite{Hor71}, Eskin\cite{Eskin},
Seeger, Sogge and Stein\cite{SSS91},
Tao\cite{Tao}, 
 Miyachi \cite{Miyachi}, Peral\cite{Peral}, Asada and Fujiwara\cite{AF}, Fujiwara\cite{Fuji}, Kumano-go\cite{Kumano-go}, Coriasco and Ruzhansky \cite{CoRu1,CoRu2}, Ruzhansky and Sugimoto \cite{RuzSugi01,RuzSugi02,RuzSugi03,RuzSugi}, Ruzhansky \cite{M. Ruzhansky}, and Ruzhansky and Wirth \cite{RWirth}.

A fundamental problem in the theory of Fourier integral operators  is that of classifying the interplay between the properties of a symbol and the properties of its  associated Fourier integral operator.

In this paper our main goal  is to give, in terms of symbol criteria and with  simple proofs,  characterizations for the $r$-nuclearity of Fourier integral operators on Lebesgue spaces. Let us mention that this problem has been considered in the case of pseudo-differential operators  by several authors. However, the obtained results belong to one of two possible approaches. The first ones, are sufficient conditions on the symbol  trough of summability  conditions with the attempt of  studying  the distribution of the spectrum for the corresponding pseudo-differential operators. The second ones, provide roughly speaking, a decomposition for the symbols associated to nuclear operators, in terms of the Fourier transform, where the spatial variables and the momentum variables can be analyzed separately. Nevertheless, in both cases the results can be applied to obtain Grothendieck-Lidskii's formulae on the summability of eigenvalues when the operators are considered acting in $L^p$ spaces. 

Necessary conditions for the $r$-nuclearity of pseudo-differential operators in the compact setting can be summarized as follows. The nuclearity and the $2/3$-nuclearity of pseudo-differential operators on the circle $\mathbb{S}^1$ and on the lattice $\mathbb{Z}$ can be found in Delgado and Wong \cite{DW}. Later, the $r$-nuclearity of pseudo-differential operators was extensively developed on arbitrary compact Lie groups and  on (closed) compact manifolds by Delgado and Ruzhansky in the works \cite{DR,DR1,DR3,DR5,DRboundedvariable2} and by the author in \cite{Cardona}; other conditions can be found  in the works \cite{DRboundedvariable,DRB,DRB2}.  Finally, the subject was treated for  compact manifolds with boundary by Delgado, Ruzhansky, and Tokmagambetov in \cite{DRTk}. 

On the other hand, characterizations for nuclear operators in terms of decomposition  of the symbol trough of the Fourier transform were investigated by Ghaemi,  Jamalpour Birgani, and  Wong in \cite{Ghaemi,Ghaemi2,Majid} for $\mathbb{S}^1,\mathbb{Z}$ and also for arbitrary compact and Hausdorff groups. Finally the subject has been considered for pseudo-multipliers associated to  the harmonic oscillator (which can be qualified as pseudo-differential operators according to the Ruzhansky-Tokmagambetov calculus when the reference operators is the quantum harmonic oscillator) in the works of the author \cite{BarrazaCardona2,Cardonabrief,BarrazaCardona}.

\subsection{Nuclear Fourier integral operators} In order to present our main result we recall the notion of nuclear operators.  By following the classical reference  Grothendieck \cite{GRO}, we  recall that a densely defined linear operator $T:D(T)\subset E\rightarrow F$  (where $D(T)$ is the domain of $T,$ and $E,F$ are choose to be Banach spaces) extends to a  $r$-nuclear operator from $E$ into $F$, if
there exist  sequences $(e_n ')_{n\in\mathbb{N}_0}$ in $ E'$ (the dual space of $E$) and $(y_n)_{n\in\mathbb{N}_0}$ in $F$ such that, the discrete representation
\begin{equation}\label{nuc}
Tf=\sum_{n\in\mathbb{N}_0} e_n'(f)y_n,\,\,\, \textnormal{ with }\,\,\,\sum_{n\in\mathbb{N}_0} \Vert e_n' \Vert^r_{E'}\Vert y_n \Vert^r_{F}<\infty,
\end{equation} holds true for all $f\in D(T).$
\noindent The class of $r-$nuclear operators is usually endowed with the natural semi-norm
\begin{equation}
n_r(T):=\inf\left\{ \left\{\sum_n \Vert e_n' \Vert^r_{E'}\Vert y_n \Vert^r_{F}\right\}^{\frac{1}{r}}: T=\sum_n e_n'\otimes y_n \right\}
\end{equation}
\noindent and, if $r=1$, $n_1(\cdot)$ is a norm and we obtain the ideal of nuclear operators. In addition, when $E=F$ is a Hilbert space and  $r=1$ the definition above agrees with that of   trace class operators. For the case of Hilbert spaces $H$, the set of $r$-nuclear operators agrees with the Schatten-von Neumann class of order $r$ (see Pietsch  \cite{P,P2}).\\
\\
In order to characterize the $r$-nuclearity of Fourier integral operators on $\mathbb{R}^n$, we will use (same that in the references mentioned above)   Delgado's characterization  (see \cite{D2}), for  nuclear integral  operators on Lebesgue spaces defined in  $\sigma$-finite measure spaces, which in this case will  be applied to $L^p(\mathbb{R}^n)$-spaces. Consequently, we will prove that  $r$-nuclear Fourier integral operators defined as in \eqref{definition}  have a nuclear trace given by
\begin{eqnarray}
\textnormal{Tr}(F)=\int\limits_{\mathbb{R}^{n}  }\int\limits_{\mathbb{R}^{n}} e^{i\phi(x,\xi)-i2\pi x\cdot \xi}a(x,\xi)dx\,d\xi.
\end{eqnarray}

In this paper our main result  is the following theorem.

\begin{theorem}\label{MainTheorem}
 Let  $0<r\leq 1.$ Let $a(\cdot,\cdot)$ be a symbol such that $a(x,\cdot)\in L^1_{loc}(\mathbb{R}^n),$ $a.e.w.,$  $x\in \mathbb{R}^n.$ Let $2\leq p_1<\infty,$   $1\leq p_2<\infty,$  and let $F$ be the Fourier integral operator associated to $a(\cdot,\cdot).$ Then, $F:L^{p_1}(\mathbb{R}^n)\rightarrow L^{p_2}(\mathbb{R}^n)$ is $r$-nuclear, if and only if,  the symbol $a(\cdot,\cdot)$  admits a decomposition of the form
\begin{equation}\label{symboldecompositionT}
a(x,\xi)=e^{-i\phi(x,\xi)}\sum_{k=1}^{\infty}h_{k}(x)(\mathscr{F}^{-1}{g}_k)(\xi),\,\,\,a.e.w.,\,\,(x,\xi),
\end{equation} where   $\{g_k\}_{k\in\mathbb{N}}$ and $\{h_k\}_{k\in\mathbb{N}}$ are sequences of functions satisfying 
\begin{equation}
\sum_{k=0}^{\infty}\Vert g_k\Vert^r_{L^{p_1'}}\Vert h_{k}\Vert^r_{L^{p_2}}<\infty.
\end{equation}
\end{theorem}

The previous result is an analogue of the main  results proved in   Ghaemi, Jamalpour Birgani, and  Wong \cite{Ghaemi,Ghaemi2}, Jamalpour Birgani \cite{Majid}, and Cardona and Barraza \cite{BarrazaCardona}. Theorem \ref{MainTheorem}, can be used for understanding the properties of the corresponding symbols in Lebesgue spaces. Moreover, we obtain the following result as a consequence of   Theorem \ref{MainTheorem}.

\begin{theorem}\label{decaying}
 Let $a(\cdot,\cdot)$ be a symbol such that $a(x,\cdot)\in L^1_{loc}(\mathbb{R}^n),$ $a.e.w.,$  $x\in \mathbb{R}^n.$ Let $2\leq p_1<\infty,$   $1\leq p_2<\infty,$  and let $F$ be the Fourier integral operator associated to $a(\cdot,\cdot).$ If $F:L^{p_1}(\mathbb{R}^n)\rightarrow L^{p_2}(\mathbb{R}^n)$ is nuclear, then $a(x,\xi)\in L^{p_2}_xL^{p_1}_\xi(\mathbb{R}^n\times \mathbb{R}^n)\cap L^{p_1}_\xi L^{p_2}_x(\mathbb{R}^n\times \mathbb{R}^n) ,$ this means that
 \begin{equation}
     \Vert a(x,\xi)\Vert_{ L^{p_2}_xL^{p_1}_\xi(\mathbb{R}^n\times \mathbb{R}^n),}:=\left(\int\limits_{\mathbb{R}^n} \left(\int\limits_{\mathbb{R}^n}|a(x,\xi)|^{p_2}dx\right)^{\frac{p_1}{p_2}}d\xi\right)^{\frac{1}{p_1}}<\infty,
 \end{equation} and 
 \begin{equation}
     \Vert a(x,\xi)\Vert_{ L^{p_1}_\xi L^{p_2}_x(\mathbb{R}^n\times \mathbb{R}^n),}:=\left(\int\limits_{\mathbb{R}^n} \left(\int\limits_{\mathbb{R}^n}|a(x,\xi)|^{p_1}d\xi \right)^{\frac{p_2}{p_1}}dx\right)^{\frac{1}{p_2}}<\infty.
 \end{equation}
\end{theorem}

Sufficient conditions in order that  pseudo-differential operators in $L^2(\mathbb{R}^n)$ can be extended to (trace class) nuclear operators  are well known. Let us recall that  the Weyl-quantization of a distribution $\sigma\in \mathscr{S}'(\mathbb{R}^{2n})$ is the pseudo-differential operator defined by
\begin{equation}\label{weyl}
Af(x)\equiv \sigma^\omega(x,D_x)f(x)=\int\limits_{\mathbb{R}^{n}}\int\limits_{\mathbb{R}^{n}}e^{i 2\pi (x-y)\cdot \xi}\sigma\left(\frac{x+y}{2},\xi\right)f(y)dyd\xi.
\end{equation}
As it is well known  $\sigma=\sigma_A(\cdot,\cdot)\in L^{1}(\mathbb{R}^{2n}),$  implies that $A:L^2\rightarrow L^2$ is class trace, and   $A:L^2\rightarrow L^2$ is Hilbert-Schmidt if and only if $\sigma_A\in L^2(\mathbb{R}^{2n}).$ In the framework of the Weyl-H\"ormander calculus of operators $A$ associated to symbols $\sigma$ in the $S(m,g)$-classes (see \cite{Hor2}), there exist two remarkable results. The first one, due to Lars H\"ormander, which asserts that  $\sigma_A\in S(m,g)$ and $\sigma\in L^{1}(\mathbb{R}^{2n})$  implies that $A:L^2\rightarrow L^2$ is a trace class operator. The second one, due to L. Rodino and F. Nicola expresses that  $\sigma_A \in S(m,g)$ and $m\in L^{1}_{w},$  (the weak-$L^1$ space), implies that $A:L^2\rightarrow L^2$ is Dixmier traceable \cite{Rod-Nic}. Moreover, an open conjecture by Rodino and Nicola (see \cite{Rod-Nic}) says that $\sigma_A\in L^{1}_{w}(\mathbb{R}^{2n})$ gives an operator $A$ with finite Dixmier trace. General properties for pseudo-differential operators on Schatten-von Neumann classes can be found in Buzano and Toft \cite{BuzanoToft}.

As an application of Theorem \ref{MainTheorem} to the Weyl quantization we present the following theorem.

\begin{theorem}\label{MainTheorem2}
 Let  $0<r\leq 1.$ Let $a(\cdot,\cdot)$ be a differentiable symbol. Let $2\leq p_1<\infty,$   $1\leq p_2<\infty,$  and let $a^\omega(x,D_x)$ be the Weyl quantization of the symbol $a(\cdot,\cdot).$ Then, $a^\omega(x,D_x):L^{p_1}(\mathbb{R}^n)\rightarrow L^{p_2}(\mathbb{R}^n)$ is $r$-nuclear, if and only if,  the symbol $a(\cdot,\cdot)$  admits a decomposition of the form
\begin{equation}
a(x,\xi)=\sum_{k=1}^{\infty}\int\limits_{\mathbb{R}^{n}} e^{-i2\pi z\cdot \xi } h_{k}\left(x+\frac{z}{2}\right) {g}_k\left(x-\frac{z}{2} \right) dz,\,\,\,a.e.w.,\,\,(x,\xi),
\end{equation} where   $\{g_k\}_{k\in\mathbb{N}}$ and $\{h_k\}_{k\in\mathbb{N}}$ are sequences of functions satisfying 
\begin{equation}
\sum_{k=0}^{\infty}\Vert g_k\Vert^r_{L^{p_1'}}\Vert h_{k}\Vert^r_{L^{p_2}}<\infty.
\end{equation}
\end{theorem}
\begin{remark}
Let us recall that the Wigner transform of two complex functions $h,g$ on $\mathbb{R}^n,$ is formally defined as
\begin{equation}
\mathscr{W}(h,g)(x,\xi):=\int\limits_{\mathbb{R}^{n}} e^{-i2\pi z\cdot \xi } h\left(x+\frac{z}{2}\right) \overline{g}\left(x-\frac{z}{2} \right) dz,\,\,\,a.e.w.,\,\,(x,\xi).
\end{equation} With a such definition in mind, if $2\leq p_1<\infty,$ $1\leq p_2<\infty,$ under the hypothesis of Theorem \ref{MainTheorem2}, $a^\omega(x,D_x):L^{p_1}(\mathbb{R}^n)\rightarrow L^{p_2}(\mathbb{R}^n),$ is $r$-nuclear, if and only if,  the symbol $a(\cdot,\cdot)$  admits a decomposition (defined trough of the {  \textit{Wigner transform}}) of the type
\begin{equation}
a(x,\xi)=\sum_{k=1}^{\infty}\mathscr{W}(h_k,\overline{g}_k)(x,\xi),\,\,\,a.e.w.,\,\,(x,\xi),
\end{equation} where   $\{g_k\}_{k\in\mathbb{N}}$ and $\{h_k\}_{k\in\mathbb{N}}$ are sequences of functions satisfying 
\begin{equation}
\sum_{k=0}^{\infty}\Vert g_k\Vert^r_{L^{p_1'}}\Vert h_{k}\Vert^r_{L^{p_2}}<\infty.
\end{equation}
\end{remark}
 The proof of our main  result (Theorem \ref{MainTheorem}) will be presented in Section \ref{Rnn} as well as the proof of Theorem \ref{MainTheorem2}. The nuclearity of Fourier integral operators on the lattice  $\mathbb{Z}^n$ and on  compact Lie groups  will be discussed in Section \ref{ZnGLie} as well as some trace formulae for FIOs on the -dimensional torus $\mathbb{T}^n=\mathbb{R}^n/\mathbb{Z}^n$ and the unitary special group $\textnormal{SU}(2)$. Finally, in Section \ref{FIOcompacthomomani} we consider the nuclearity of FIOs on arbitrary compact homogeneous manifolds and we discuss the case of the complex projective space $\mathbb{C}\mathbb{P}^2$. In this setting, we will prove analogues for the theorems  \ref{MainTheorem} and  \ref{decaying} in every context mentioned above.

\section{Symbol criteria for nuclear Fourier integral operators}\label{Rnn}
\subsection{Characterization of nuclear FIOs}

In this section we prove our main result for Fourier integral operators $F$ defined as in \eqref{definition}. Our criteria will be formulated in terms of the symbols $a.$ First, let us observe that every FIO $F$ has a integral representation with kernel  $K(x,y).$ In fact, straightforward computation shows us that
\begin{equation}\label{kernelpseudo}
Ff(x) :=\int_{\mathbb{R}^n}K(x,y)f(y)dy,\,\,\,\,\,\,\,\,\,\,\,\,\,\,\,\,\,\,\,
\end{equation} where $$ K(x,y):=\int_{\mathbb{R}^n}e^{i\phi(x,\xi)-i2\pi y\cdot \xi}a(x,\xi)d\xi, $$
for every  $f\in\mathscr{D}(\mathbb{R}^n).$ In order to analyze the $r$-nuclearity of the Fourier integral operator  $F$ we will studying  its kernel $K,$ by using as a fundamental tool, the following theorem (see J. Delgado \cite{Delgado,D2}).
\begin{theorem}\label{Theorem1} Let us consider $1\leq p_1,p_2<\infty,$ $0<r\leq 1$ and let $p_i'$ be such that $\frac{1}{p_i}+\frac{1}{p_i'}=1.$ Let $(X_1,\mu_1)$ and $(X_2,\mu_2)$ be $\sigma$-finite measure spaces. An operator $T:L^{p_1}(X_1,\mu_1)\rightarrow L^{p_2}(X_2,\mu_2)$ is $r$-nuclear if and only if there exist sequences $(h_k)_k$ in $L^{p_2}(\mu_2),$ and $(g_k)$ in $L^{p_1'}(\mu_1),$ such that
\begin{equation}
\sum_{k}\Vert h_k\Vert_{L^{p_2}}^r\Vert g_k\Vert^r_{L^{p_1'} } <\infty,\textnormal{        and        }Tf(x)=\int\limits_{X_1}(\sum_k h_{k}(x)g_k(y))f(y)d\mu_1(y),\textnormal{   a.e.w. }x,
\end{equation}
for every $f\in {L^{p_1}}(\mu_1).$ In this case, if $p_1=p_2,$ and $\mu_1=\mu_2,$ $($\textnormal{see Section 3 of} \cite{Delgado}$)$ the nuclear trace of $T$ is given by
\begin{equation}\label{trace1}
\textnormal{Tr}(T):=\int\limits_{X_1}\sum_{k}g_{k}(x)h_{k}(x)d\mu_1(x).
\end{equation}
\end{theorem}
\begin{remark}\label{remarkmain} Given $f\in L^1(\mathbb{R}^n),$ define its Fourier transform by
\begin{equation}\mathscr{F}f(\xi):=\int_{\mathbb{R}^n}e^{-i2\pi x\cdot \xi}{f}(x)dx.
\end{equation} If we consider a function $f,$ such that $f\in L^1(\mathbb{R}^n)$ with $\mathscr{F}f\in L^1(\mathbb{R}^n),$ the Fourier inversion formula gives
\begin{equation}{f}(x)=\int_{\mathbb{R}^n}e^{i2\pi x\cdot \xi}\mathscr{F}f(\xi)d\xi.
\end{equation} Moreover, the Hausdorff-Young inequality $\Vert \mathscr{F}f \Vert_{L^{p'}}\leq \Vert f \Vert_{L^{p}}$ (with $\Vert \mathscr{F}f \Vert_{L^2}=\Vert f \Vert_{L^2}$)
shows that the Fourier transform is a well defined operator on $L^p,$ $1<p\leq2$.
\end{remark}
Now, we prove our main theorem.
\begin{theorem}
 Let  $0<r\leq 1.$ Let $a(\cdot,\cdot)$ be a symbol such that $a(x,\cdot)\in L^1_{loc}(\mathbb{R}^n),$ $a.e.w.,$  $x\in \mathbb{R}^n.$ Let $2\leq p_1<\infty,$   $1\leq p_2<\infty,$  and let $F$ be the Fourier integral operator associated to $a(\cdot,\cdot).$ Then, $F:L^{p_1}(\mathbb{R}^n)\rightarrow L^{p_2}(\mathbb{R}^n)$ is $r$-nuclear, if and only if,  the symbol $a(\cdot,\cdot)$  admits a decomposition of the form
\begin{equation}\label{symboldecompositionT}
a(x,\xi)=e^{-i\phi(x,\xi)}\sum_{k=1}^{\infty}h_{k}(x)(\mathscr{F}^{-1}{g}_k)(\xi),\,\,\,a.e.w.,\,\,(x,\xi),
\end{equation} where   $\{g_k\}_{k\in\mathbb{N}}$ and $\{h_k\}_{k\in\mathbb{N}}$ are sequences of functions satisfying 
\begin{equation}
\sum_{k=0}^{\infty}\Vert g_k\Vert^r_{L^{p_1'}}\Vert h_{k}\Vert^r_{L^{p_2}}<\infty.
\end{equation}
\end{theorem}
\begin{proof} Let us assume that $F$ is a Fourier integral operator as in \eqref{definition} with associated symbol $a$.  Let us assume that $F:L^{p_1}(\mathbb{R}^n)\rightarrow L^{p_2}(\mathbb{R}^n)$ is $r$-nuclear. Then there exist sequences $h_{k}$ in $L^{p_2}$ and $g_k$ in $L^{p_1'}$ satisfying
\begin{equation}
Ff(x)=\int_{\mathbb{R}^n}\left(\sum_{k=1}^{\infty}h_k(x)g_{k}(y)\right)f(y)dy,\,\, f\in L^{p_1},
\end{equation} with 
\begin{equation}\label{deco2}
\sum_{k=0}^{\infty}\Vert g_k\Vert^r_{L^{p_1'}}\Vert h_{k}\Vert^r_{L^{p_2}}<\infty.
\end{equation}

Then there exist sequences $h_{k}$ in $L^{p_2}$ and $g_k$ in $L^{p_1'}$ satisfying
\begin{equation}
Ff(x)=\int_{\mathbb{R}^n}\left(\sum_{k=1}^{\infty}h_k(x)g_{k}(y)\right)f(y)dy,\,\, f\in L^{p_1},
\end{equation} with 
\begin{equation}\label{deco2}
\sum_{k=0}^{\infty}\Vert g_k\Vert^r_{L^{p_1'}}\Vert h_{k}\Vert^r_{L^{p_2}}<\infty.
\end{equation}

For all $z\in\mathbb{R}^n,$ let us consider the set  $B(z,r)$, i.e., the euclidean  ball centered at $z$ with radius $r>0.$ Let us denote by $|B(z,r)|$ the Lebesgue measure of $B(z,r).$ Let us choose $\xi_{0}\in \mathbb{R}^n$ and $r>0.$  If we define $\delta_{\xi_0}^{r}:=|B(\xi_0,r)|^{-1}\cdot 1_{B(\xi_0,r)},$ where $ 1_{B(\xi,r)}$ is the characteristic function of the ball $B(\xi_0,r),$ the condition $2\leq p_1<\infty,$ together with the Hausdorff-Young inequality gives
\begin{equation}
\Vert \mathscr{F}^{-1}(\delta_{\xi_0}^r)\Vert_{L^{p_1}}= \Vert \mathscr{F}^{-1}(\delta_{\xi_0}^r)\Vert_{L^{(p_1')'}} \leq \Vert \delta_{\xi_0}^r \Vert_{L^{p_1'}}=1.
\end{equation}
So, for every $r>0$ and $\xi_0\in\mathbb{R}^n,$ the function $\mathscr{F}^{-1}\delta_{\xi_0}^r\in L^{p_1}(\mathbb{R}^n)=\textnormal{Dom}(F), $  and we get,
\begin{align*}
F(\mathscr{F}^{-1}\delta_{\xi_0}^r)(x) &=\int_{\mathbb{R}^n}\left(\sum_{k=1}^{\infty}h_k(x)g_k(y)\right) \mathscr{F}^{-1}\delta_{\xi_0}^r(y) dy.
\end{align*}
Taking into account that $K(x,y)= \sum_{k=1}^{\infty}h_k(x)g_k(y)\in L^{1}(\mathbb{R}^{2n})$ (see, e.g.,  Lemma 3.1 of \cite{DRboundedvariable}), that $\Vert \mathscr{F}^{-1}\delta_{\xi_0}^r\Vert_{L^\infty}\leq \Vert \delta_{\xi_0}^r \Vert_{L^1}=1,$ and that (in view of the Lebesgue Differentiation Theorem)
\begin{equation}
\lim_{r\rightarrow 0^+}\mathscr{F}^{-1}\delta_{\xi_0}^r(x)=\lim_{r\rightarrow 0^+}\frac{1}{|B(\xi_0,r)|}\int_{B(\xi_0,r)}e^{i2\pi x\cdot\xi}d\xi=e^{i2\pi x\cdot\xi_0},
\end{equation}
an application of the Dominated Convergence Theorem gives
\begin{equation}
\lim_{r\rightarrow 0^+}F(\mathscr{F}^{-1}\delta_{\xi_0} ^r)(x)=\int_{\mathbb{R}^n}\left(\sum_{k=1}^{\infty}h_k(x)g_k(y)\right) e^{i2\pi y\cdot \xi_0} dy=\sum_{k=1}^{\infty}h_k(x)(\mathscr{F}^{-1}g_k)(\xi_0).
\end{equation}
In fact, for $a.e.w.$ $x\in \mathbb{R}^n,$
\begin{align*}
    \left|\left(\sum_{k=1}^{\infty}h_k(x)g_k(y)\right)(\mathscr{F}^{-1}\delta_{\xi_0} ^r)(y)\right| &=  \left|\left(\sum_{k=1}^{\infty}h_k(x)g_k(y)\right)\frac{1}{|B(\xi_0,r)|}\int_{B(\xi_0,r)}e^{i2\pi y\cdot\xi}d\xi \right|\\
    &\leq \left|\sum_{k=1}^{\infty}h_k(x)g_k(y) \right|=|K(x,y)|.
\end{align*}  Because, $K\in L^1(\mathbb{R}^{2n}),$ and the function $\kappa(x,y):=|K(x,y)|$ is non-negative  on the product space $\mathbb{R}^{2n},$ by the Fubinni theorem applied to  positive functions, the $L^1(\mathbb{R}^{2n})$-norm of $K$ can be computed from iterated integrals as
\begin{equation}
\int\int|K(x,y)|dy,dx=\int\left(\int  |K(x,y)|dy\right)dx=\int\left(\int  |K(x,y)|dx\right)dy.
\end{equation} By  Tonelly theorem, for $a.e.w.$ $x\in\mathbb{R}^n,$ the function $\kappa(x,\cdot)=|K(x,\cdot)|\in L^1(\mathbb{R}^n).$ Now, by the dominated convergence theorem, we have
\begin{align*}
   & \lim_{r\rightarrow 0^+}F(\mathscr{F}^{-1}\delta_{\xi_0} ^r)(x)\\
   &=\lim_{r\rightarrow 0^+}\int_{\mathbb{R}^n}K(x,y)\mathscr{F}^{-1}\delta_{\xi_0} ^r(y)dy= \int_{\mathbb{R}^n}K(x,y) \lim_{r\rightarrow 0^{+}}\mathscr{F}^{-1}\delta_{\xi_0} ^r(y)dy\\
    &= \int_{\mathbb{R}^n}K(x,y) e^{i2\pi y\xi_0}dy=\lim_{\ell\rightarrow \infty} \int_{|y|\leq \ell}K(x,y)\cdot e^{i2\pi y\xi_0} dy\\
    &=\lim_{\ell\rightarrow \infty} \int_{\mathbb{R}^n  }\left(\sum_{k=1}^{\infty}h_k(x)g_k(y)\right)\cdot e^{i2\pi y\xi_0} \cdot 1_{\{ |y|\leq \ell \} }  \cdot dy
    .
\end{align*} Now, from Lemma 3.4-$(d)$ in  \cite{DRboundedvariable},
\begin{align*}
   \lim_{\ell\rightarrow \infty} &\int_{\mathbb{R}^n  }\left(\sum_{k=1}^{\infty}h_k(x)g_k(y)\right)\cdot e^{i2\pi y\xi_0} \cdot 1_{\{ |y|\leq \ell \} }  \cdot dy\\
   &=\lim_{\ell,m\rightarrow \infty} \int_{\mathbb{R}^n  }\left(\sum_{k=1}^{m}h_k(x)g_k(y)\right)\cdot e^{i2\pi y\xi_0} \cdot 1_{\{ |y|\leq \ell \} }  \cdot dy \\
   &=\lim_{\ell,m\rightarrow \infty}\sum_{k=1}^{m}   h_k(x)\int_{\mathbb{R}^n  }g_k(y)\cdot e^{i2\pi y\xi_0} \cdot 1_{\{ |y|\leq \ell \} }  \cdot dy \\
   &=\lim_{\ell,m\rightarrow \infty}\sum_{k=1}^{m}   h_k(x)\int_{|y|\leq \ell }g_k(y)\cdot e^{i2\pi y\xi_0} dy\\
   &=\sum_{k=1}^{\infty}h_k(x)(\mathscr{F}^{-1}g_k)(\xi_0).
\end{align*}

On the other hand, if we compute $F(\mathscr{F}^{-1}\delta_{\xi_0}^r)$  from the definition \eqref{definition}, we have
\begin{align*}
F(\mathscr{F}^{-1}\delta_{\xi_0}^r)(x) =\frac{1}{|B(\xi_0,r)|}\int_{B(\xi_0,r)}e^{i\phi(x,\xi)}a(x,\xi)d\xi.
\end{align*} From the hypothesis that $a(x,\cdot)\in L^{1}_{loc}(\mathbb{R}^n)$ for $a.e.w$ $x\in \mathbb{R}^n,$ the Lebesgue Differentiation theorem gives 
\begin{equation}
\lim_{r\rightarrow 0^+}  F(\mathscr{F}^{-1}\delta_{\xi_0}^r)=e^{i\phi(x,\xi_0)}a(x,\xi_0).
\end{equation}  Consequently, we deduce the identity
\begin{equation}\label{deco}
e^{i\phi(x,\xi_0)}a(x,\xi_0)=\sum_{k=1}^{\infty}h_k(x)(\mathscr{F}^{-1}{g}_{k})(\xi_0),
\end{equation} which in turn is equivalent to
\begin{equation}\label{deco}
a(x,\xi_0)=e^{-i\phi(x,\xi_0)}\sum_{k=1}^{\infty}h_k(x)(\mathscr{F}^{-1}{g}_{k})(\xi_0).
\end{equation}
So, we have proved the first part of the theorem. Now, if we assume that the symbol $a$ of the FIO $F$ satisfies the decomposition formula \eqref{deco} for fixed sequences  $h_{k}$ in $L^{p_2}$ and $g_k$ in $L^{p_1'}$ satisfying \eqref{deco2}, then from \eqref{definition} we can write
\begin{align*}
Ff(x) &=\int_{\mathbb{R}^n}e^{\phi(x,\xi)}a(x,\xi)\mathscr{F}f(\xi)d\xi=\int_{\mathbb{R}^n}\sum_{k=1}^{\infty}h_k(x)(\mathscr{F}^{-1}
{g}_{k})(\xi)(\mathscr{F}{f})(\xi)d\xi\\
&=\int_{\mathbb{R}^n}\sum_{k=1}^{\infty}h_k(x)\int_{\mathbb{R}^n}e^{i2\pi y\xi} g_k(y)dy(\mathscr{F}{f})(\xi)d\xi\\
&= \int_{\mathbb{R}^n}\left(\sum_{k=1}^{\infty}h_k(x)g_k(y)\right)\left(\int_{\mathbb{R}^n}e^{i2\pi y\xi} (\mathscr{F}{f})(\xi)d\xi \right)dy\\
&= \int_{\mathbb{R}^n}\left(\sum_{k=1}^{\infty}h_k(x)g_k(y)\right)f(y)dy,
\end{align*}
where in the last line we have used the Fourier inversion formula. So, by  Delgado Theorem (Theorem \ref{Theorem1}) we finish the proof.
\end{proof}

\begin{proof}[Proof of Theorem \ref{decaying}] Let $a(\cdot,\cdot)$ be a symbol such that $a(x,\cdot)\in L^1_{loc}(\mathbb{R}^n),$ $a.e.w.,$  $x\in \mathbb{R}^n.$ Let $2\leq p_1<\infty,$   $1\leq p_2<\infty,$  and let $F$ be the Fourier integral operator associated to $a(\cdot,\cdot).$ If $F:L^{p_1}(\mathbb{R}^n)\rightarrow L^{p_2}(\mathbb{R}^n)$ is nuclear, then Theorem \ref{MainTheorem} guarantees the decomposition
$$a(x,\xi)=e^{-i\phi(x,\xi)}\sum_{k=1}^{\infty}h_{k}(x)(\mathscr{F}^{-1}{g}_k)(\xi),\,\,\,a.e.w.,\,\,(x,\xi),
$$ where   $\{g_k\}_{k\in\mathbb{N}}$ and $\{h_k\}_{k\in\mathbb{N}}$ are sequences of functions satisfying 
\begin{equation}
\sum_{k=0}^{\infty}\Vert g_k\Vert_{L^{p_1'}}\Vert h_{k}\Vert_{L^{p_2}}<\infty.
\end{equation} So, if we take the $L^{p_2}_x$-norm, we have,
\begin{align*}
\Vert a(x,\xi) \Vert_{L^{p_2}_x} &= \left\Vert e^{-i\phi(x,\xi)}\sum_{k=1}^{\infty}h_{k}(x)(\mathscr{F}^{-1}{g}_k)(\xi)  \right\Vert_{L^{p_2}_x}\\
&=\left\Vert \sum_{k=1}^{\infty}h_{k}(x)(\mathscr{F}^{-1}{g}_k)(\xi)  \right\Vert_{L^{p_2}_x}\\
&\leq \sum_{k=1}^{\infty}\Vert h_{k}\Vert_{L^{p_2}}|(\mathscr{F}^{-1}{g}_k)(\xi)|.
\end{align*} Now, if we use the Hausdorff-Young inequality, we deduce, $\Vert  \mathscr{F}^{-1}{g}_k\Vert_{L^{p_1}}\leq \Vert  {g}_k\Vert_{L^{p_1'}}.$
Consequently,
\begin{align*}
    \Vert a(x,\xi)\Vert_{ L^{p_2}_xL^{p_1}_\xi(\mathbb{R}^n\times \mathbb{R}^n),} &=\left(\int\limits_{\mathbb{R}^n} \left(\int\limits_{\mathbb{R}^n}|a(x,\xi)|^{p_2}dx\right)^{\frac{p_1}{p_2}}d\xi\right)^{\frac{1}{p_1}}\\
    &\leq\left\Vert  \sum_{k=1}^{\infty}\Vert h_{k}\Vert_{L^{p_2}}|(\mathscr{F}^{-1}{g}_k)(\xi)| \right\Vert_{L^{p_1}_\xi} \\
    &\leq \sum_{k=1}^{\infty}\Vert h_{k}\Vert_{L^{p_2}}\Vert\mathscr{F}^{-1}{g}_k\Vert_{L^{p_1}}\\
     &\leq \sum_{k=1}^{\infty}\Vert h_{k}\Vert_{L^{p_2}}\Vert{g}_k\Vert_{L^{p_1'}}<\infty.
\end{align*}
In an analogous way we can prove that
  $$ \Vert a(x,\xi)\Vert_{ L^{p_1}_\xi L^{p_2}_x(\mathbb{R}^n\times \mathbb{R}^n)}\leq \sum_{k=1}^{\infty}\Vert h_{k}\Vert_{L^{p_2}}\Vert{g}_k\Vert_{L^{p_1'}}<\infty. $$
Thus, we finish the proof.
\end{proof}

\subsection{The nuclear trace for FIOs on $\mathbb{R}^n$} If we choose a $r$-nuclear  operator  $T:E\rightarrow E$, $0<r\leq 1,$ with the Banach space $E$ satisfying the Grothendieck approximation property (see Grothendieck\cite{GRO}),  then 
there exist (a nuclear decomposition)  sequences $(e_n ')_{n\in\mathbb{N}_0}$ in $ E'$ (the dual space of $E$) and $(y_n)_{n\in\mathbb{N}_0}$ in $E$ satisfying 
\begin{equation}\label{nuc2}
Tf=\sum_{n\in\mathbb{N}_0} e_n'(f)y_n,\,\,\,\,\,f\in E,
\end{equation}
and \begin{equation}
\sum_{n\in\mathbb{N}_0} \Vert e_n' \Vert^r_{E'}\Vert y_n \Vert^r_{F}<\infty.
\end{equation}
In this case the nuclear trace of $T$ is (a  well-defined functional) given by
$
\textnormal{Tr}(T)=\sum_{n\in\mathbb{N}^n_0}e_n'(f_n).
$
Because $L^p$-spaces have the Grothendieck approximation property and as consequence we can compute the nuclear trace of every $r$-nuclear pseudo-multipliers. We will compute it from Delgado Theorem (Theorem \ref{Theorem1}). For to do so, let us consider a $r-$nuclear Fourier integral operator  $F:L^{p}(\mathbb{R}^n)\rightarrow L^{p}(\mathbb{R}^n),$ $2\leq p<\infty.$  If $a$ is the symbol associated to $F,$ in view of \eqref{symboldecompositionT}, we have
\begin{align*}
\int\limits_{\mathbb{R}^{2n}} & e^{i\phi(x,\xi)-2\pi ix\cdot \xi}a(x,\xi)d\xi\,dx = \int\limits_{\mathbb{R}^n}\int\limits_{\mathbb{R}^n}e^{i\phi(x,\xi)-2\pi ix\cdot \xi} e^{-i\phi(x,\xi)}\sum_{k=1}^{\infty}h_k(x)(\mathscr{F}{g}_{k})(-\xi)d\xi\,dx\\
&= \int\limits_{\mathbb{R}^n}\sum_{k=1}^{\infty}h_k(x)\int\limits_{\mathbb{R}^n}e^{-2\pi ix\cdot \xi} (\mathscr{F}^{-1}{g}_{k})(\xi)d\xi\,dx\\
&= \int\limits_{\mathbb{R}^n}\sum_{k=1}^{\infty}h_k(x)g_{k}(x)dx=\textnormal{Tr}(F).
\end{align*}
So, we obtain the trace formula
\begin{equation}
\textnormal{Tr}(F)=\int\limits_{\mathbb{R}^n}\int\limits_{\mathbb{R}^n}   e^{i\phi(x,\xi)-2\pi ix\cdot \xi}a(x,\xi)d\xi\,dx.
\end{equation}

Now, in order to determinate a relation with the eigenvalues of $F$ we recall that, the nuclear trace of an $r$-nuclear operator  on a Banach space coincides with the spectral trace, provided that $0<r\leq \frac{2}{3}.$ For $\frac{2}{3}\leq r\leq 1$ we recall the following result (see \cite{O}).
\begin{theorem} Let $T:L^p(\mu)\rightarrow L^p(\mu)$ be a $r$-nuclear operator as in \eqref{nuc2}. If $\frac{1}{r}=1+|\frac{1}{p}-\frac{1}{2}|,$ then, 
\begin{equation}
\textnormal{Tr}(T):=\sum_{n\in\mathbb{N}^n_0}e_n'(f_n)=\sum_{n}\lambda_n(T)
\end{equation}
where $\lambda_n(T),$ $n\in\mathbb{N}$ is the sequence of eigenvalues of $T$ with multiplicities taken into account. 
\end{theorem}
As an immediate consequence of the preceding theorem, if the FIO  $F:L^p(\mathbb{R}^n)\rightarrow L^p(\mathbb{R}^n)$ is  $r$-nuclear, the relation $\frac{1}{r}=1+|\frac{1}{p}-\frac{1}{2}|$ implies, 
\begin{equation}
\textnormal{Tr}(F)=\int\limits_{\mathbb{R}^n}\int\limits_{\mathbb{R}^n}   e^{i\phi(x,\xi)-2\pi ix\cdot \xi}a(x,\xi)d\xi\,dx=\sum_{n}\lambda_n(T),
\end{equation}
where $\lambda_n(T),$ $n\in\mathbb{N}$ is the sequence of eigenvalues of $F$ with multiplicities taken into account.\\

\subsection{Characterization of nuclear pseudo-differential operators defined by the Weyl-H\"ormander quantization}

As it was mentioned in the introduction, the Weyl-quantization of a distribution $\sigma\in \mathscr{S}'(\mathbb{R}^{2n})$ is the pseudo-differential operator defined by
\begin{equation}\label{weyl}
\sigma^\omega(x,D_x)f(x)=\int\limits_{\mathbb{R}^{n}}\int\limits_{\mathbb{R}^{n}}e^{i 2\pi (x-y)\cdot \xi}\sigma\left(\frac{x+y}{2},\xi\right)f(y)dyd\xi.
\end{equation}
There exist relations between pseudo-differential operators associated to the classical quantization
\begin{equation}
\sigma(x,D_x)f(x)=\int\limits_{\mathbb{R}^{n}}\int\limits_{\mathbb{R}^{n}}e^{i 2\pi (x-y)\cdot \xi}\sigma\left(x,\xi\right)f(y)dyd\xi=\int\limits_{\mathbb{R}^{n}}e^{i 2\pi x\cdot \xi}\sigma\left(x,\xi\right)(\mathscr{F}f)(\xi)d\xi,
\end{equation}
or in a more general setting,  $\tau$-quantizations defined for every $0<\tau\leq 1,$ by the integral expression,
\begin{equation}
\sigma^{\tau}(x,D_x)f(x)=\int\limits_{\mathbb{R}^{n}}\int\limits_{\mathbb{R}^{n}}e^{i 2\pi (x-y)\cdot \xi}\sigma\left(\tau x+(1-\tau)y,\xi\right)f(y)dyd\xi,
\end{equation}
(with $\tau=\frac{1}{2}$ corresponding to the Weyl-H\"ormander quantization) as it can be viewed in the following proposition (see Delgado \cite{Del2006}).
\begin{proposition}\label{relationbquant}
Let $a, b\in  \mathscr{S}'(\mathbb{R}^{2n}). $ Then, $a^\tau(x,D_x)=b^{\tau'}(x,D_x)$ if and only if
\begin{equation}
a(x,\xi)=\int\limits_{\mathbb{R}^{n}}\int\limits_{\mathbb{R}^{n}} e^{-i2\pi (\xi-\eta)z}b(x+(\tau'-\tau)z,\eta)dzd\eta
\end{equation} provided that $0< \tau,\tau'\leq 1.$
\end{proposition}
\begin{theorem}\label{MainTheoremtau}
 Let  $0<r\leq 1.$ Let $a(\cdot,\cdot)$ be a differentiable symbol. Let $2\leq p_1<\infty,$   $1\leq p_2<\infty,$  and let $a^\omega(x,D_x)$ be the Weyl-H\"ormander quantization of the symbol $a(\cdot,\cdot).$ Then, $a^\omega(x,D_x):L^{p_1}(\mathbb{R}^n)\rightarrow L^{p_2}(\mathbb{R}^n)$ is $r$-nuclear, if and only if,  the symbol $a(\cdot,\cdot)$  admits a decomposition of the form
\begin{equation}\label{weyl-deco}
a(x,\xi)=\sum_{k=1}^{\infty}\int\limits_{\mathbb{R}^{n}} e^{-i2\pi z\cdot \xi } h_{k}(x+(1-\tau)z) {g}_k(x- \tau z) dz,\,\,\,a.e.w.,\,\,(x,\xi),
\end{equation} where   $\{g_k\}_{k\in\mathbb{N}}$ and $\{h_k\}_{k\in\mathbb{N}}$ are sequences of functions satisfying 
\begin{equation}
\sum_{k=0}^{\infty}\Vert g_k\Vert^r_{L^{p_1'}}\Vert h_{k}\Vert^r_{L^{p_2}}<\infty.
\end{equation}
\end{theorem}
\begin{proof}
Let us assume that $a^\tau(x,D_x)$ is $r$-nuclear from $L^{p_1}(\mathbb{R}^n)$ into $L^{p_2}(\mathbb{R}^n)$. By Proposition \ref{relationbquant}, $a^\tau(x,D_x)=b(x,D_x)$ where 
$$a(x,\xi)=\int\limits_{\mathbb{R}^{n}}\int\limits_{\mathbb{R}^{n}} e^{-i2\pi (\xi-\eta)z}b(x+(1-\tau)z,\eta)dzd\eta.$$ By Theorem \ref{MainTheorem} applied to $\phi(x,\xi)= 2\pi x\cdot \xi,$ and taking into account that $b(x,D_x)$ is $r$-nuclear,  there exist sequences $h_{k}$ in $L^{p_2}$ and $g_k$ in $L^{p_1'}$ satisfying
\begin{equation}\label{bxxi}
b(x,\xi)=e^{-i2\pi x\cdot \xi}\sum_{k=1}^{\infty}h_{k}(x)(\mathscr{F}^{-1}{g}_k)(\xi),\,\,\,a.e.w.,\,\,(x,\xi),
\end{equation} with 
\begin{equation}\label{deco2}
\sum_{k=0}^{\infty}\Vert g_k\Vert^r_{L^{p_1'}}\Vert h_{k}\Vert^r_{L^{p_2}}<\infty.
\end{equation}
So, we have
\begin{align*}
a(x,\xi)=\int\limits_{\mathbb{R}^{n}}\int\limits_{\mathbb{R}^{n}} e^{-i2\pi (\xi-\eta)z-i2\pi(x+(1-\tau)z)\eta}  \left( \sum_{k=1}^{\infty}h_{k}(x+(1-\tau)z)(\mathscr{F}^{-1}{g}_k)(\eta)\right)  dzd\eta.
\end{align*}
Since
\begin{align*}  -i2\pi (\xi-\eta)z-i2\pi  &(x+(1-\tau)z)\cdot\eta \\
&=-i2\pi \xi\cdot z+i2\pi \eta \cdot z-i2\pi x\cdot \eta-i2\pi (1-\tau)z\cdot \eta \\
&=-i2\pi \xi\cdot z-i2\pi x\cdot \eta+i2\pi \tau z\cdot \eta ,
\end{align*} we have
\begin{align*}
a(x,\xi)&=\sum_{k=1}^{\infty}\int\limits_{\mathbb{R}^{n}}\int\limits_{\mathbb{R}^{n}}e^{ -i2\pi \xi\cdot z-i2\pi x\cdot \eta+i2\pi \tau z\cdot \eta  } h_{k}(x+(1-\tau)z)(\mathscr{F}^{-1}{g}_k)(\eta)dzd\eta\\
&=\sum_{k=1}^{\infty}\int\limits_{\mathbb{R}^{n}} e^{ -i2\pi \xi\cdot z} h_{k}(x+(1-\tau)z) \int\limits_{\mathbb{R}^{n}}e^{ -i2\pi( x- \tau z)\cdot \eta  } (\mathscr{F}^{-1}{g}_k)(\eta)d\eta dz\\
&=\sum_{k=1}^{\infty}\int\limits_{\mathbb{R}^{n}} e^{-i2\pi z\cdot \xi } h_{k}(x+(1-\tau)z) {g}_k(x- \tau z) dz.
\end{align*}
So, we have proved the first part of the characterization. On the other hand, if we assume \eqref{weyl-deco}, then
 $$a(x,\xi)=\int\limits_{\mathbb{R}^{n}}\int\limits_{\mathbb{R}^{n}} e^{-i2\pi (\xi-\eta)z}b(x+(1-\tau)z,\eta)dzd\eta,$$
where $b(x,\xi)$ is defined as in \eqref{bxxi}. So, from Theorem \ref{MainTheorem} we deduce that $b(x,D_x)$ is $r$-nuclear and from the equality $a^\tau(x,D_x)=b(x,D_x)$ we deduce the $r$-nuclearity of $a^\tau(x,D_x).$ The proof is complete.
\end{proof}
\begin{remark}
Let us observe that from Theorem \ref{MainTheoremtau} with $\tau=1/2,$ we deduce the Theorem \ref{MainTheorem2} mentioned in the introduction.
\end{remark}

\section{Characterizations of Fourier integral operators on $\mathbb{Z}^n$  and  arbitrary compact Lie groups}\label{ZnGLie}

\subsection{FIOs on $\mathbb{Z}^n$} In this subsection we characterize those Fourier integral operators on $\mathbb{Z}^n$ (the set of points in $\mathbb{R}^n$ with integral coordinates) admitting nuclear extensions on Lebesgue spaces. Now we define pseudo-differential operators and discrete  Fourier integral operators on $\mathbb
{Z}^n.$ The discrete Fourier transform of $f\in \ell^1(\mathbb{Z})$ is defined by
\begin{equation}
(\mathscr{F}_{\mathbb{Z}^n}f)(\xi)=\sum_{m\in\mathbb{Z}^n}e^{-i2\pi m\cdot \xi}f(m),\,\,\xi\in [0,1]^n.
\end{equation}
The Fourier inversion formula gives
\begin{equation}
f(m)=\int_{[0,1]^n}e^{i2\pi m\cdot \xi} (\mathscr{F}_{\mathbb{Z}^n}f)(\xi)d\xi,\,\,\,m\in\mathbb{Z}^n.
\end{equation} 
In this setting pseudo-differential operators on $\mathbb{Z}^n$ are defined by the integral form
\begin{equation}\label{pseudo}
  t_m f(n'):=\int_{[0,1]^n} e^{i2\pi n'\cdot\xi}m(n',\xi)(\mathscr{F}_{\mathbb{Z}^n}f)(\xi)d\xi,\,\,\,f\in \ell^1(\mathbb{Z}^n),\,n'\in\mathbb{Z}^n.
\end{equation}
These operators were introduced by Molahajloo in \cite{m}. However, the fundamental work Botchway L., Kibiti G., Ruzhansky M., \cite{RuzhanskyZn} provide a symbolic calculus and other properties for these operators on $\ell^p$-spaces. In particular,   Fourier integral operators on $\mathbb{Z}^n$ were  defined in such reference as integral operators of the form
\begin{equation}\label{pseudo}
  \mathfrak{f}_{a,\phi} f(n'):=\int_{[0,1]^n} e^{i\phi( n',\xi)}a(n',\xi)(\mathscr{F}_{\mathbb{Z}^n}f)(\xi)d\xi,\,\,\,f\in \ell^1(\mathbb{Z}^n),\,n'\in\mathbb{Z}^n.
\end{equation}
Our main tool in the characterization of nuclear FIOs on $\mathbb{Z}^n$ is the following result due to Jamalpour Birgani \cite{Majid}.
\begin{theorem}\label{Majid2}
 Let  $0<r\leq 1,$  $1\leq p_1<\infty,$   $1\leq p_2<\infty,$  and let $t_m$ be the pseudo-differential  operator associated to  the symbol $m(\cdot,\cdot).$ Then, $t_m:\ell^{p_1}(\mathbb{Z}^n)\rightarrow \ell^{p_2}(\mathbb{Z}^n)$ is $r$-nuclear, if and only if,  the symbol $m(\cdot,\cdot)$  admits a decomposition of the form
\begin{equation}\label{symboldecompositionT22222}
m(n',\xi)=e^{-i2\pi n'\xi}\sum_{k=1}^{\infty}h_{k}(n')(\mathscr{F}_{\mathbb{Z}^n}{{g}}_k)(-\xi),\,\,\,a.e.w.,\,\,(n',\xi),
\end{equation} where   $\{g_k\}_{k\in\mathbb{N}}$ and $\{h_k\}_{k\in\mathbb{N}}$ are sequences of functions satisfying 
\begin{equation}
\sum_{k=0}^{\infty}\Vert g_k\Vert^r_{\ell^{p_1'}}\Vert h_{k}\Vert^r_{\ell^{p_2}}<\infty.
\end{equation}
\end{theorem}
As a consequence of the previous result, we give a simple proof for our characterization.

\begin{theorem}\label{FIOSZncardona}
 Let  $0<r\leq 1,$  $1\leq p_1<\infty,$   $1\leq p_2<\infty,$  and let $\mathfrak{f}_{a,\phi}$ be the Fourier integral operator associated to the phase function $\phi$ and to  the symbol $a(\cdot,\cdot).$ Then, $\mathfrak{f}_{a,\phi}:\ell^{p_1}(\mathbb{Z}^n)\rightarrow \ell^{p_2}(\mathbb{Z}^n)$ is $r$-nuclear, if and only if,  the symbol $a(\cdot,\cdot)$  admits a decomposition of the form
\begin{equation}\label{symboldecompositionT2222}
a(n',\xi)=e^{-i\phi(n',\xi)}\sum_{k=1}^{\infty}h_{k}(x)(\mathscr{F}_{\mathbb{Z}^n}{{g}}_k)(-\xi),\,\,\,a.e.w.,\,\,(n',\xi),
\end{equation} where   $\{g_k\}_{k\in\mathbb{N}}$ and $\{h_k\}_{k\in\mathbb{N}}$ are sequences of functions satisfying 
\begin{equation}
\sum_{k=0}^{\infty}\Vert g_k\Vert^r_{\ell^{p_1'}}\Vert h_{k}\Vert^r_{\ell^{p_2}}<\infty.
\end{equation}
\end{theorem}
\begin{proof}
Let us write the operator  $\mathfrak{f}_{a,\phi}$  as
\begin{equation}  \mathfrak{f}_{a,\phi}f(x)=\int_{[0,1]^n} e^{i\phi( n',\xi)}a(n',\xi)(\mathscr{F}_{\mathbb{Z}^n}f)(\xi)d\xi=\int_{[0,1]^n} e^{i2\pi  n'\cdot\xi}m(n',\xi)(\mathscr{F}_{\mathbb{Z}^n}f)(\xi)d\xi
\end{equation} where $m(n',\xi)=e^{i\phi(n', \xi)-i2\pi n'\cdot \xi}a(n',\xi).$ So, the discrete Fourier integral operator $ \mathfrak{f}_{a,\phi}$ coincides with the discrete pseudo-differential operator $t_m$ with symbol $m.$ By using Theorem \ref{Ghaemi2}, the operator $\mathfrak{f}_{a,\phi}=t_m:\ell^{p_1}(\mathbb{Z}^n)\rightarrow \ell^{p_2}(\mathbb{Z}^n) $ is $r$-nuclear, if and only if,  the symbol $m(\cdot,\cdot)$  admits a decomposition of the form
\begin{equation}\label{symboldecompositionT22}
m(n',\xi)=e^{-i2\pi n'\xi}\sum_{k=1}^{\infty}h_{k}(n')(\mathscr{F}_{\mathbb{Z}^n}{{g}}_k)(-\xi),\,\,\,a.e.w.,\,\,(n',\xi),
\end{equation} where   $\{g_k\}_{k\in\mathbb{N}}$ and $\{h_k\}_{k\in\mathbb{N}}$ are sequences of functions satisfying 
\begin{equation}
\sum_{k=0}^{\infty}\Vert g_k\Vert^r_{\ell^{p_1'}}\Vert h_{k}\Vert^r_{\ell^{p_2}}<\infty.
\end{equation}
Let us note that from the definition of $m$ we have
$$ e^{i\phi(n',\xi)-i2\pi n'\cdot \xi}a(n',\xi)= e^{-i2\pi n'\cdot \xi}\sum_{k=1}^{\infty}h_{k}(n')(\mathscr{F}_{ \mathbb{Z}^n }{{g}}_k)(-\xi),\,\,\,a.e.w.,\,\,(n',\xi),$$
which in turn, is equivalent to
\begin{equation} a(n',\xi)= e^{-i\phi( n', \xi)}\sum_{k=1}^{\infty}h_{k}(n')(\mathscr{F}_{ \mathbb{Z}^n }{{g}}_k)(-\xi).
\end{equation} Thus, the proof is complete.
\end{proof}
\begin{remark} The nuclear trace of a nuclear discrete pseudo-differential operator on $\mathbb{Z}^n,$ $t_m:\ell^{p}(\mathbb{Z}^n)\rightarrow \ell^{p}(\mathbb{Z}^n),$ $1\leq p<\infty,$ can be computed according to the formula
\begin{equation}
\textnormal{Tr}(t_m)=\sum_{n'\in\mathbb{Z}^n}\int\limits_{[0,1]^n}m(n',\xi)d\xi.
\end{equation} From the proof of the previous criterion, we have that $\mathfrak{f}_{a,\phi}=t_m$ where $m(n',\xi)=e^{\phi(n',\xi)-i2\pi n'\xi}a(n',\xi)$ and consequently, if $\mathfrak{f}_{a,\phi}:\ell^{p}(\mathbb{Z}^n)\rightarrow \ell^{p}(\mathbb{Z}^n),$ $1\leq p<\infty,$ is $r$-nuclear, its nuclear trace is given by
\begin{equation}
\textnormal{Tr}(\mathfrak{f}_{a,\phi})=\sum_{n'\in\mathbb{Z}^n}\int\limits_{[0,1]^n} e^{\phi(n',\xi)-i2\pi n'\xi}a(n',\xi)d\xi.
\end{equation}
\end{remark}

Now, we present an application of the previous result. 
\begin{theorem}
 Let $2\leq p_1<\infty,$ and    $1\leq p_2<\infty.$ If $\mathfrak{f}_{a,\phi}:\ell^{p_1}(\mathbb{Z}^n)\rightarrow \ell^{p_2}(\mathbb{Z}^n)$ is nuclear, then $a(n',\xi)\in \ell^{p_2}_{n'}L^{p_1}_\xi(\mathbb{Z}^n\times \mathbb{T}^n)\cap L^{p_1}_\xi \ell^{p_2}_{n'}(\mathbb{Z}^n\times \mathbb{T}^n) ,$ this means that
 \begin{equation}
     \Vert a(n',\xi)\Vert_{ \ell^{p_2}_{n'}L^{p_1}_\xi(\mathbb{Z}^n\times \mathbb{T}^n),}:=\left(\int\limits_{\mathbb{T}^n} \left(\sum\limits_{n'\in \mathbb{Z}^n}|a(n',\xi)|^{p_2}\right)^{\frac{p_1}{p_2}}d\xi\right)^{\frac{1}{p_1}}<\infty,
 \end{equation} and 
 \begin{equation}
     \Vert a(n',\xi)\Vert_{ L^{p_1}_\xi \ell^{p_2}_{n'}(\mathbb{Z}^n\times \mathbb{T}^n),}:=\left(\sum\limits_{n'\in\mathbb{Z}^n} \left(\int\limits_{\mathbb{T}^n}|a(x,\xi)|^{p_1}d\xi \right)^{\frac{p_2}{p_1}}\right)^{\frac{1}{p_2}}<\infty.
 \end{equation}
 \end{theorem}
The proof is only an adaptation of the proof that we have done for Theorem \ref{decaying}. We only need to use a discrete Hausdorff-Young inequality. In this case, we use \begin{equation}
    \Vert \mathscr{F}_{\mathbb{Z}^n}g_k\Vert_{L^{p_1}(\mathbb{T}^n)}\leq \Vert g_k\Vert_{\ell^{p_1'}(\mathbb{Z}^n)}.
\end{equation}

\subsection{FIOs on compact Lie groups}
In this subsection we characterize nuclear Fourier integral operators on compact Lie groups. Although the results presented are valid for arbitrary Hausdorff and compact groups, we restrict our attention to Lie groups taking into account their differentiable structure, which in our case could be give potential applications of our results to the understanding on the spectrum of certain operators associated to differential problems.

Let us consider a compact Lie group $G$ with Lie algebra $\mathfrak{g}$. We will equip $G$ with the Haar measure $\mu_{G}$. The following identities follow from the Fourier transform on $G$
 $$ (\mathscr{F}_{G}\varphi)(\xi)\equiv \widehat{\varphi}(\xi)=\int_{G}\varphi(x)\xi(x)^*dx,\,\,\,\,\,\,\,\,\,\,\,\,\,\,\,\, \varphi(x)=\sum_{[\xi]\in \widehat{G}}d_{\xi}\text{Tr}(\xi(x)\widehat{\varphi}(\xi)) ,$$
and the Peter-Weyl Theorem on $G$ implies the Plancherel identity on $L^2(G),$
$$ \Vert \varphi \Vert_{L^2(G)}= \left(\sum_{[\xi]\in \widehat{G}}d_{\xi}\text{Tr}(\widehat{\varphi}(\xi)\widehat{\varphi}(\xi)^*) \right)^{\frac{1}{2}}=\Vert  \widehat{\varphi}\Vert_{ L^2(\widehat{G} ) } .$$
\noindent Notice that, since $\Vert A \Vert^2_{HS}=\text{Tr}(AA^*)$, the term within the sum is the Hilbert-Schmidt norm of the matrix $\widehat{\varphi}(\xi)$. Any linear operator $A$ on $G$ mapping $C^{\infty}(G)$ into $\mathcal{D}'(G)$ gives rise to a {\em matrix-valued global (or full) symbol} $\sigma_{A}(x,\xi)\in \mathbb{C}^{d_\xi \times d_\xi}$ given by
\begin{equation}
\sigma_A(x,\xi)=\xi(x)^{*}(A\xi)(x),
\end{equation}
which can be understood from the distributional viewpoint. Then it can be shown that the operator $A$ can be expressed in terms of such a symbol as \cite{Ruz}
\begin{equation}\label{mul}Af(x)=\sum_{[\xi]\in \widehat{G}}d_{\xi}\text{Tr}[\xi(x)\sigma_A(x,\xi)\widehat{f}(\xi)]. 
\end{equation}
So, if $\Phi:G\times \widehat{G}\rightarrow \cup_{ [\xi]\in \widehat{G} }\textnormal{GL}(d_\xi)$ is a measurable function (the phase function), and $a:G\times \widehat{G}\rightarrow \cup_{ [\xi]\in \widehat{G} }\mathbb{C}^{d_\xi\times d_\xi}$ is a distribution on $G\times \widehat{G},$ the Fourier integral operator $F=F_{\Phi,a}$ associated to the symbol $a(\cdot,\cdot)$ and to the phase function $\Phi$ is defined by the  Fourier series operator
\begin{equation}\label{FIOsG}Ff(x)=\sum_{[\xi]\in \widehat{G}}d_{\xi}\text{Tr}[\Phi(x,\xi)a(x,\xi)\widehat{f}(\xi)]. 
\end{equation}
In order to present our main result for Fourier integral operators, we recall the following criterion (see Ghaemi, Jamalpour Birgani, Wong \cite{Ghaemi2}). 
\begin{theorem}\label{Ghaemi2}
 Let  $0<r\leq 1,$  $1\leq p_1<\infty,$   $1\leq p_2<\infty,$  and let $A$ be the pseudo-differential operator associated to  the symbol $\sigma_A(\cdot,\cdot).$ Then, $A:L^{p_1}(G)\rightarrow L^{p_2}(G)$ is $r$-nuclear, if and only if,  the symbol $\sigma_A(\cdot,\cdot)$  admits a decomposition of the form
\begin{equation}\label{symboldecompositionT22}
\sigma_A(x,\xi)=\xi(x)^*\sum_{k=1}^{\infty}h_{k}(x)(\mathscr{F}_{G}{\overline{g}}_k)(\xi)^*,\,\,\,a.e.w.,\,\,(x,\xi),
\end{equation} where   $\{g_k\}_{k\in\mathbb{N}}$ and $\{h_k\}_{k\in\mathbb{N}}$ are sequences of functions satisfying 
\begin{equation}
\sum_{k=0}^{\infty}\Vert g_k\Vert^r_{L^{p_1'}}\Vert h_{k}\Vert^r_{L^{p_2}}<\infty.
\end{equation}
\end{theorem}
As a consequence of the previous criterion, we give a simple proof for our characterization. 

\begin{theorem}\label{FIOSGcardona}
 Let  $0<r\leq 1,$  $1\leq p_1<\infty,$   $1\leq p_2<\infty,$  and let $F$ be the Fourier integral operator associated to the phase function $\Phi$ and to  the symbol $a(\cdot,\cdot).$ Then, $F:L^{p_1}(G)\rightarrow L^{p_2}(G)$ is $r$-nuclear, if and only if,  the symbol $a(\cdot,\cdot)$  admits a decomposition of the form
\begin{equation}\label{symboldecompositionT22}
a(x,\xi)=\Phi(x,\xi)^{-1}\sum_{k=1}^{\infty}h_{k}(x)(\mathscr{F}_{G}{\overline{g}}_k)(\xi)^*,\,\,\,a.e.w.,\,\,(x,\xi),
\end{equation} where   $\{g_k\}_{k\in\mathbb{N}}$ and $\{h_k\}_{k\in\mathbb{N}}$ are sequences of functions satisfying 
\begin{equation}
\sum_{k=0}^{\infty}\Vert g_k\Vert^r_{L^{p_1'}}\Vert h_{k}\Vert^r_{L^{p_2}}<\infty.
\end{equation}
\end{theorem}
\begin{remark}
For the proof we use use the characterization of $r$-nuclear pseudo-differential operators mentioned above. However, this result will be generalized in the next section to arbitrary compact homogeneous manifolds.
\end{remark}
\begin{proof}
Let us observe that the Fourier integral operator $F,$ can be written as
\begin{equation}Ff(x)=\sum_{[\xi]\in \widehat{G}}d_{\xi}\text{Tr}[\Phi(x,\xi)a(x,\xi)\widehat{f}(\xi)]=\sum_{[\xi]\in \widehat{G}}d_{\xi}\text{Tr}[\xi(x)\sigma_A(x,\xi)\widehat{f}(\xi)],
\end{equation} where $\sigma_A(x,\xi)=\xi(x)^*\Phi(x,\xi)a(x,\xi).$ So, the Fourier integral operator $F$ coincides with the pseudo-differential operator $A$ with symbol $\sigma_A.$ In view of Theorem \ref{Ghaemi2}, the operator $F=A:L^{p_1}(G)\rightarrow L^{p_2}(G) $ is $r$-nuclear, if and only if,  the symbol $\sigma_A(\cdot,\cdot)$  admits a decomposition of the form
\begin{equation}\label{symboldecompositionT22}
\sigma_A(x,\xi)=\xi(x)^*\sum_{k=1}^{\infty}h_{k}(x)(\mathscr{F}_{G}{\overline{g}}_k)(\xi)^*,\,\,\,a.e.w.,\,\,(x,\xi),
\end{equation} where   $\{g_k\}_{k\in\mathbb{N}}$ and $\{h_k\}_{k\in\mathbb{N}}$ are sequences of functions satisfying 
\begin{equation}
\sum_{k=0}^{\infty}\Vert g_k\Vert^r_{L^{p_1'}}\Vert h_{k}\Vert^r_{L^{p_2}}<\infty.
\end{equation}
Let us note that from the definition of $\sigma_A$ we have
$$ \xi(x)^*\Phi(x,\xi)a(x,\xi)= \xi(x)^*\sum_{k=1}^{\infty}h_{k}(x)(\mathscr{F}_{G}{\overline{g}}_k)(\xi)^*,\,\,\,a.e.w.,\,\,(x,\xi),$$
which is equivalent to
\begin{equation} a(x,\xi)= \Phi(x,\xi)^{-1}\sum_{k=1}^{\infty}h_{k}(x)(\mathscr{F}_{G}{\overline{g}}_k)(\xi)^*,\,\,\,a.e.w.,\,\,(x,\xi).
\end{equation} Thus, we finish the proof.
\end{proof}
\begin{remark} The nuclear trace of a $r$-nuclear pseudo-differential operator on $G,$ $A:L^{p}(G)\rightarrow L^{p}(G),$ $1\leq p<\infty,$ can be computed according to the formula
\begin{equation}
\textnormal{Tr}(A)=\int\limits_{G}\sum_{[\xi]\in \widehat{G}}d_\xi\textnormal{Tr}[\sigma_A(x,\xi)]dx.
\end{equation} From the proof of the previous theorem, we have that $F=A$ where $\sigma_A(x,\xi)=\xi(x)^{*}\Phi(x,\xi)a(x,\xi)$ and consequently, if $F:L^{p}(G)\rightarrow L^{p}(G),$ $1\leq p<\infty,$ is $r$-nuclear, its nuclear trace is given by
\begin{equation}
\textnormal{Tr}(F)=\int\limits_{G}\sum_{[\xi]\in \widehat{G}}d_\xi\textnormal{Tr}[\xi(x)^{*}\Phi(x,\xi)a(x,\xi)]dx.
\end{equation}
\end{remark}

Now, we illustrate the results above with some examples.
\begin{example}$(\textnormal{The torus}).$ Let us consider the $n$-dimensional torus $G=\mathbb{T}^n:=\mathbb{R}^n/\mathbb{Z}^n$  and its unitary dual $\widehat{\mathbb{T}}^n:=\{e_\ell:\ell\in\mathbb{Z}^n\},$ $e_\ell(x):=e^{i2\pi\ell\cdot x},$ $x\in\mathbb{T}^n.$ By following Ruzhansky  and Turunen \cite{Ruz}, a Fourier integral operator $F$ associated to the phase function $\phi:\mathbb{T}^n\times \widehat{\mathbb{T}}^n\rightarrow \mathbb{C} ,$ and to the symbol $a:\mathbb{T}^n\times \widehat{\mathbb{T}}^n\rightarrow \mathbb{C},$ is defined according to the rule,
\begin{equation}
F\varphi(x)=\sum_{e_\ell \in \widehat{\mathbb{T}}^n  }e^{i\phi(x,e_\ell)}a(x,e_\ell)\widehat{\varphi}(e_\ell),\,\,x\in\mathbb{T}^n,
\end{equation}
where $\widehat{\varphi}(e_\ell)=\int_{\mathbb{T}^n}f(x)e_{\ell}(x)dx,$ is the Fourier transform of $f$ at $e_\ell.$ If we identify $\widehat{\mathbb{T}}^n$ with $ \mathbb{Z}^n,$ and we define $a(x,\ell):=a(x,e_\ell),$ and $\widehat{\varphi}(\ell):=\widehat{\varphi}(e_\ell)$ we give the more familiar expression for $F,$
\begin{equation}
F\varphi(x)=\sum_{\ell \in \mathbb{Z}^n  }e^{i\phi(x,\ell)}a(x,\ell)\widehat{\varphi}(\ell),\,\,x\in\mathbb{T}^n.
\end{equation} Now, by using Theorem \ref{FIOSGcardona}, $F$ is $r$-nuclear, $0<r\leq 1,$ if and only if,  the symbol $a(\cdot,\cdot)$  admits a decomposition of the form
\begin{equation}
a(x,\ell)=e^{-i\phi(x,\ell)}\sum_{k=1}^{\infty}h_{k}(x)\overline{(\widehat{\overline{g}}_k)(\ell)},\,\,\,a.e.w.,\,\,(x,\ell),
\end{equation} where   $\{g_k\}_{k\in\mathbb{N}}$ and $\{h_k\}_{k\in\mathbb{N}}$ are sequences of functions satisfying 
\begin{equation}
\sum_{k=0}^{\infty}\Vert g_k\Vert^r_{L^{p_1'}}\Vert h_{k}\Vert^r_{L^{p_2}}<\infty.
\end{equation} The last condition have been proved for pseudo-differential operators in \cite{Ghaemi}. In this case, the nuclear trace of $F$ can be written as
\begin{equation}
\textnormal{Tr}(F)=\int_{\mathbb{T}^n}\sum_{\ell\in\mathbb{Z}^n}e^{\phi(x,\ell)-i2\pi x\cdot \ell}a(x,\ell)dx.
\end{equation}
\end{example}

\begin{example}$(\textnormal{The group SU(2)}).$ Let us consider the group $\textnormal{SU}(2)\cong \mathbb{S}^3$ consinting of those orthogonal matrices $A$ in $\mathbb{C}^{2\times 2},$ with $\det(A)=1$.   We recall that the unitary dual of $\textnormal{SU}(2)$ (see \cite{Ruz}) can be identified as
\begin{equation}
\widehat{\textnormal{SU}}(2)\equiv \{ [t_{l}]:2l\in \mathbb{N}, d_{l}:=\dim t_{l}=(2l+1)\}.
\end{equation}
There are explicit formulae for $t_{l}$ as
functions of Euler angles in terms of the so-called Legendre-Jacobi polynomials, see \cite{Ruz}. A Fourier integral operator $F$ associated to the phase function $\Phi:\textnormal{SU}(2)\times \widehat{\textnormal{SU}}(2)\rightarrow \cup_{\ell\in \frac{1}{2}\mathbb{N}_0 } \textnormal{GL}(2\ell+1),$ and to the symbol $a:\textnormal{SU}(2)\times \widehat{\textnormal{SU}}(2)\rightarrow \cup_{\ell\in \frac{1}{2}\mathbb{N}_0 } \mathbb{C}^{(2\ell+1)\times (2\ell+1)},$ is defined as,
\begin{equation}
F\varphi(x)=\sum_{  [t_\ell] \in \widehat{\textnormal{SU}}(2)  }  (2\ell+1)\textnormal{Tr}[\Phi(x,t_\ell)a(x,t_\ell)\widehat{\varphi}(t_\ell)],\,\,x\in\textnormal{SU}(2),
\end{equation}
where $$\widehat{\varphi}(e_\ell)=\int_{  \textnormal{SU}(2) }f(x)t_{\ell}(x)dx\in \mathbb{C}^{(2\ell+1)\times (2\ell+1)},\,\,\ell\in \frac{1}{2}\mathbb{N}_0,$$ is the Fourier transform of $f$ at $t_\ell.$ As in the case of the $n$-dimensional torus, if we identify $\widehat{\textnormal{SU}}(2)$ with $ \frac{1}{2}\mathbb{N}_0,$ and we define $a(x,\ell):=a(x,t_\ell),$ and $\widehat{\varphi}(\ell):=\widehat{\varphi}(t_\ell)$ we can write
\begin{equation}
F\varphi(x)=\sum_{  \ell \in \frac{1}{2}\mathbb{N}_0  }  (2\ell+1)\textnormal{Tr}[\Phi(x,\ell)a(x,\ell)\widehat{\varphi}(\ell)],\,\,x\in\textnormal{SU}(2).
\end{equation} Now, by using Theorem \ref{FIOSGcardona}, $F$ is $r$-nuclear, $0<r\leq 1,$ if and only if,  the symbol $a(\cdot,\cdot)$  admits a decomposition of the form
\begin{equation}
a(x,\ell)=\Phi(x,\ell)^{-1}\sum_{k=1}^{\infty}h_{k}(x)\overline{(\widehat{\overline{g}}_k)(\ell)},\,\,\,a.e.w.,\,\,(x,\ell),
\end{equation} where   $\{g_k\}_{k\in\mathbb{N}}$ and $\{h_k\}_{k\in\mathbb{N}}$ are sequences of functions satisfying 
\begin{equation}
\sum_{k=0}^{\infty}\Vert g_k\Vert^r_{L^{p_1'}}\Vert h_{k}\Vert^r_{L^{p_2}}<\infty.
\end{equation} The last condition have been proved for pseudo-differential operators in \cite{Ghaemi2} on arbitrary Hausdorff and compact groups. In this case, in an analogous expression to the presented above for $\mathbb{R}^n,$ $\mathbb{Z}^n,$ and $\mathbb{T}^n,$ the nuclear trace of $F$ can be written as
\begin{equation}
\textnormal{Tr}(F)=\int\limits_{\textnormal{SU}(2)}\sum_{\ell\in \frac{1}{2}\mathbb{N}_0}(2\ell+1)\textnormal{Tr}[t_\ell(A)^{*}\Phi(A,\ell)a(A,\ell)]dA.
\end{equation} By using the diffeomorphism 
$\varrho:\textnormal{SU}(2)\rightarrow \mathbb{S}^3,$ defined by
\begin{equation}\varrho(A)=x:=(x_1,x_2,x_3,x_4),\,\,\,\textnormal{for}\,\,\,\,
A=\begin{bmatrix}
    x_1+ix_2       & x_3+ix_4  \\
    -x_3+ix_4       & x_1-ix_2 
\end{bmatrix}, 
\end{equation} we have
\begin{align*}
\textnormal{Tr}(F)&=\int\limits_{\textnormal{SU}(2)}\sum_{\ell\in \frac{1}{2}\mathbb{N}_0}(2\ell+1)\textnormal{Tr}[t_\ell(A)^{*}\Phi(A,\ell)a(A,\ell)]dA\\
&=\int\limits_{ \mathbb{S}^3  }\sum_{\ell\in \frac{1}{2}\mathbb{N}_0}(2\ell+1)\textnormal{Tr}[t_\ell(\varrho^{-1}(x))^{*}\Phi( \varrho^{-1}(x)  ,\ell)a(  \varrho^{-1}(x),\ell)]d\sigma(x)\\
&=\int\limits_{ \mathbb{S}^3  }\sum_{\ell\in \frac{1}{2}\mathbb{N}_0}(2\ell+1)\textnormal{Tr}[t_\ell(x)^{*}\Phi( x  ,\ell)a(  x,\ell)]d\sigma(x),
\end{align*}where $t_\ell(\varrho^{-1}(x))=:t_\ell(x),$  $\Phi( \varrho^{-1}(x)  ,\ell)=:\Phi( x  ,\ell),$ $a(  \varrho^{-1}(x),\ell)=:a(  x,\ell)$ and    $d\sigma(x)$ denotes the surface measure on $\mathbb{S}^3.$ If we consider the parametrization of $\mathbb{S}^3$ defined by $x_1:=\cos(\frac{t}{2}),$ $x_2:=\nu,$ $x_3:=(\sin^2(\frac{t}{2})-\nu^2)^{\frac{1}{2}}\cos(s),$ $x_4:=(\sin^2(\frac{t}{2})-\nu^2)^{\frac{1}{2}}\sin(s),$ where
$$(t,\nu,s)\in D:=\{(t,\nu,s)\in\mathbb{R}^3:|\nu|\leq \sin(\frac{t}{2}),\,0\leq t,s\leq 2\pi\},$$
\end{example}then $d\sigma(x)=\sin(\frac{t}{2})dtd\nu ds,$ and 
\begin{align*}
&\textnormal{Tr}(F)\\
&=\int\limits_{0}^{2\pi}\int\limits_{0}^{2\pi}\int\limits_{  -\sin(t/2) }^{ \sin(t/2)   }\sum_{\ell\in \frac{1}{2}\mathbb{N}_0}(2\ell+1)\textnormal{Tr}[t_\ell(t,\nu,s)^{*}\Phi( (t,\nu,s)  ,\ell)a(  (t,\nu,s),\ell)]\sin(\frac{t}{2})d\nu dtds.
\end{align*}

\section{Nuclear Fourier integral operators on compact homogeneous manifolds}\label{FIOcompacthomomani}
The main goal in this section is to provide a characterization for the nuclearity of Fourier integral operators on compact homogeneous manifolds $M\cong G/K$.  Taking into account that the Peter-Weyl decompositions of $L^2(M)$ and $L^2(\mathbb{G})$ (where  $\mathbb{G}$ is a Hausdorff and compact group) have an analogue structure, we classify the nuclearity of FIOs on compact homogeneous manifolds by adapting to our case, the proof of Theorem 2.2 in \cite{Ghaemi2} where were classified those nuclear pseudo-differential operators in    $L^p(\mathbb{G})$-spaces.
\subsection{Global FIOs on compact homogeneous manifolds}
In order to present our definition for Fourier integral operators on compact homogeneous spaces, we recall some definitions on the subject. Compact homogeneous manifolds can be obtained if we consider the quotient space of a compact Lie groups $G$ with one of its closed subgroups  $K$\,\, --there exists an unique differential structure for the quotient $M:=G/K$--. Examples of compact homogeneous spaces are spheres $\mathbb{S}^n\cong \textnormal{SO}(n+1)/\textnormal{SO}(n),$ real projective spaces $\mathbb{RP}^n\cong \textnormal{SO}(n+1)/\textnormal{O}(n),$ complex projective spaces $\mathbb{CP}^n\cong\textnormal{SU}(n+1)/\textnormal{SU}(1)\times\textnormal{SU}(n)$ and more generally Grassmannians $\textnormal{Gr}(r,n)\cong\textnormal{O}(n)/\textnormal{O}(n-r)\times \textnormal{O}(r).$

Let us denote by $\widehat{G}_0$ the subset of $\widehat{G},$  of representations in $G$, that are of class I with respect to the subgroup $K$. This means that $\pi\in \widehat{G}_0$ if there exists at least one non trivial invariant vector $a$ with respect to $K,$ i.e., $\pi(h)a=a$ for every $h\in K.$ Let us denote by $B_{\pi}$ to the vector space of these invariant vectors and $k_{\pi}=\dim B_{\pi}.$ Now we follow the notion of Multipliers as in \cite{RR}. Let us consider the class of symbols $\Sigma(M),$ for $M=G/K,$ consisting of those matrix-valued functions \begin{equation}
\sigma:\widehat{G}_0\rightarrow\bigcup_{n=1}^{\infty}\mathbb{C}^{n\times n}\,\,\text{ such that }\,\,\sigma(\pi)_{ij}=0\textnormal{ for all } i,j>k_{\pi}.
\end{equation}
Following \cite{RR}, a Fourier multiplier $A$ on $M$ is a bounded operator on $L^2(M)$ such that for some $\sigma_{A}\in \Sigma(M)$ satisfies
\begin{equation}\label{Eq.MultFourHomg}
Af(x)=\sum_{\pi\in\widehat{G}_0}d_{\pi}\textnormal{Tr}(\pi(x)\sigma_{A}(\pi)\widehat{f}(\pi)), \,\,\text{ for }\,f\in C^{\infty}(M),
\end{equation}
where $\widehat{f}$ denotes the Fourier transform of the lifting $\dot{f}\in C^{\infty}(G)$ of $f$ to $G,$ given by $\dot{f}(x):=f(xK),$ $x\in G.$
\begin{remark}
For every symbols of a Fourier multipliers $A$ on $M,$ only the upper-left block in $\sigma_{A}(\pi)$ of the size $k_\pi\times k_\pi$ cannot be  the  trivial matrix zero. 
\end{remark} Now, if we consider a phase function $\Phi:M\times \widehat{G}_0\rightarrow \cup_{[\pi]\in \widehat{G}_0}\textnormal{GL}(d_\pi),$ and a distribution $a: M\times \widehat{G}_0\rightarrow \cup_{[\pi]\in \widehat{G}_0}\textnormal{GL}(d_\pi), $ the Fourier integral operator associated to $\Phi$ and to $a(\cdot,\cdot)$ is given by
\begin{equation}\label{cardonadefiunition}
F\varphi(x)=\sum_{\pi\in\widehat{G}_0}d_{\pi}\textnormal{Tr}(\Phi(x,\pi)a(x,\pi)\widehat{\varphi}(\pi)), \,\,\text{ for }\,\varphi\in C^{\infty}(M).
\end{equation}
We additionally require the condition $\sigma(x,\pi)_{ij}=0$ for $i,j>k_\pi$ for the distributional symbols considered above.  Now, if we want to characterize those $r$-nuclear FIOs we only need to follow the proof of Theorem 2.2 in \cite{Ghaemi2} where  the nuclearity of pseudo-differential operators was characterized on compact and Hausdorff groups.   Because the set $$\{\sqrt[2]{d_\pi}\pi_{ij}:1\leq i,j\leq k_\pi\}$$ provides
an orthonormal basis of $L^2(M),$ we have the relation
\begin{equation}
\int_{M}\pi_{nm}(x)\overline{\varkappa_{ij}(x)}dx=\frac{1}{d_\pi}\delta_{\pi\varkappa}\delta_{ni}\delta_{mj},\,\,\,[\pi],[\varkappa]\in \widehat{G}_0.
\end{equation} If we assume that $F:L^{p_1}(M)\rightarrow L^{p_2}(M)$ is $r$-nuclear, then we have a nuclear decomposition for its kernel, i.e., there exist sequences $h_{k}$ in $L^{p_2}$ and $g_k$ in $L^{p_1'}$ satisfying
\begin{equation}
Ff(x)=\int_{M}\left(\sum_{k=1}^{\infty}h_k(x)g_{k}(y)\right)f(y)dy,\,\, f\in L^{p_1}(M),
\end{equation} with 
\begin{equation}\label{deco2222222}
\sum_{k=0}^{\infty}\Vert g_k\Vert^r_{L^{p_1'}}\Vert h_{k}\Vert^r_{L^{p_2}}<\infty.
\end{equation}
So, we have  with $1\leq n,m\leq k_\pi,$
\begin{align*}
&F\pi_{nm}(x)\\
&=(\Phi(x,\pi)a(x,\pi))_{nm}=\int_{M}\left(\sum_{k=1}^{\infty}h_k(x)g_{k}(y)\right)\pi_{nm}(x)dx=\sum_{k=1}^{\infty}h_k(x)\overline{\widehat{ \overline{ g_k } } (\pi)}_{mn}.
\end{align*}
Consequently, if $B^t$ denotes the transpose of a matrix $B$,  we obtain
\begin{equation}
\Phi(x,\pi)a(x,\pi)=\sum_{k=1}^{\infty}h_k(x)\overline{\widehat{ \overline{ g_k } } (\pi)}^{t}=\sum_{k=1}^{\infty}h_k(x)\widehat{ \overline{ g_k } } (\pi)^*,
\end{equation}
and by considering that $\Phi(x,\pi)\in \textnormal{GL}(d_\pi)$ for every $x\in M,$ we deduce the equivalent condition,
\begin{equation}\label{M}
a(x,\pi)=\Phi(x,\pi)^{-1}\sum_{k=1}^{\infty}h_k(x)\widehat{ \overline{ g_k } } (\pi)^*.
\end{equation} On the other hand,  if we assume that the symbol $a(\cdot,\cdot)$ satisfies the condition \eqref{M} with $
\sum_{k=0}^{\infty}\Vert g_k\Vert^r_{L^{p_1'}}\Vert h_{k}\Vert^r_{L^{p_2}}<\infty, 
$ from the definition of Fourier integral operator we can write for $\varphi\in L^{p_1}(M),$
\begin{align*}
F\varphi(x) &=\sum_{\pi\in\widehat{G}_0}d_{\pi}\textnormal{Tr}(\Phi(x,\pi)a(x,\pi)\widehat{\varphi}(\pi))= \sum_{\pi\in\widehat{G}_0}d_{\pi}\textnormal{Tr}(\sum_{k=1}^{\infty}h_k(x) \widehat{ \overline{ g_k } } (\pi)^*\widehat{\varphi}(\pi))\\
& = \sum_{\pi\in\widehat{G}_0}d_{\pi}\textnormal{Tr}(\sum_{k=1}^{\infty}h_k(x)\int_{M}g_k(y)\pi(y)dy \widehat{\varphi}(\pi))\\
& = \int_{M}\sum_{k=1}^{\infty}h_k(x) g_k(y)\sum_{\pi\in\widehat{G}_0}d_{\pi}\textnormal{Tr}(\pi(y) \widehat{\varphi}(\pi))dy\\
&=\int_{M}\sum_{k=1}^{\infty}h_k(x) g_k(y)\varphi(y)dy.\\
\end{align*}
Newly, by Delgado's Theorem we obtain the $r$-nuclearity of $F.$ So, our adaptation of the proof of Theorem 2.2 in \cite{Ghaemi2}, to our case of FIOs on compact manifolds leads to the following result.
\begin{theorem}\label{compactmanifoldcardona}
Let us assume $M\cong G/K$ be a homogeneous manifold, $0<r\leq 1,$ $1\leq p_1,p_2<\infty$ and let $F$ be a Fourier integral operator as in \eqref{cardonadefiunition}. Then, $F:L^{p_1}(M)\rightarrow L^{p_2}(M)$ is $r$-nuclear if and only if,  there exist sequences $h_{k}$ in $L^{p_2}$ and $g_k$ in $L^{p_1'}$ satisfying
\begin{equation}
a(x,\pi)=\Phi(x,\pi)^{-1}\sum_{k=1}^{\infty}h_k(x)\widehat{ \overline{ g_k } } (\pi)^*,\,\,x\in G,[\pi]\in \widehat{G}_0,
\end{equation} with 
\begin{equation}\label{deco2222222111}
\sum_{k=0}^{\infty}\Vert g_k\Vert^r_{L^{p_1'}}\Vert h_{k}\Vert^r_{L^{p_2}}<\infty.
\end{equation}
\end{theorem}
Now, we will prove that the previous (abstract) characterization can be applied in order to measure the decaying of symbols in the momentum  variables. So, we will use the following formulation of Lebesgue spaces on $\widehat{G}_0:$
\begin{eqnarray}
\mathcal{M}\in \ell^p(\widehat{G}_0)\Longleftrightarrow \Vert \mathcal{M} \Vert_{\ell^p(\widehat{G}_0)}=\left( \sum_{[\pi]\in \widehat{G}_0}d_\pi k_\pi^{p(\frac{1}{p}-\frac{1}{2})}\Vert \mathcal{M}(\pi)\Vert^{p}_{\textnormal{HS}}\right)^{\frac{1}{p}}<\infty,
\end{eqnarray}
for $1\leq p<\infty.$

\begin{theorem}\label{compactmanifoldcardona2222}
Let us assume $M\cong G/K$ be a homogeneous manifold, $2\leq p_1<\infty, $ $1\leq p_2<\infty$ and let $F$ be a Fourier integral operator as in \eqref{cardonadefiunition}. If $F:L^{p_1}(M)\rightarrow L^{p_2}(M)$ is nuclear, 
then $a(x,\pi)\in  \ell^{p_1}_\pi L^{p_2}_x(M\times \widehat{G}_0) ,$ this means that
 \begin{equation}
     \Vert a(x,\pi)\Vert_{ \ell^{p_1}_\pi L^{p_2}_x(M\times \widehat{G}_0)}:=\left(\int\limits_{M} \left(\sum\limits_{[\pi]\in \widehat{G}_0}d_{\pi}k_{\pi}^{p_1(\frac{1}{p_1}-\frac{1}{2})}\Vert a(x,\pi)\Vert^{p_1}_{\textnormal{HS} }\right)^{\frac{p_2}{p_1}}dx\right)^{\frac{1}{p_2}}<\infty,
 \end{equation}
\end{theorem}provided that
\begin{eqnarray}
\Vert \Phi^{-1}\Vert_{\infty}:=\sup_{(x,[\pi])\in M\times \widehat{G}_0}\Vert \Phi(x,\pi)^{-1} \Vert_{op}<\infty.
\end{eqnarray}
\begin{proof}
Let $2\leq p_1<\infty,$   $1\leq p_2<\infty,$  and let $F$ be the Fourier integral operator associated to $a(\cdot,\cdot).$ If $F:L^{p_1}(M)\rightarrow L^{p_2}(M)$ is nuclear, then Theorem \ref{compactmanifoldcardona} guarantees the decomposition
\begin{equation}
a(x,\pi)=\Phi(x,\pi)^{-1}\sum_{k=1}^{\infty}h_k(x)\widehat{ \overline{ g_k } } (\pi)^*,\,\,x\in G,[\pi]\in \widehat{G}_0,
\end{equation} with 
\begin{equation}
\sum_{k=0}^{\infty}\Vert g_k\Vert^r_{L^{p_1'}}\Vert h_{k}\Vert^r_{L^{p_2}}<\infty.
\end{equation} So, if we take the $\ell^{p_1}_\pi$-norm, we have,
\begin{align*}
\Vert a(x,\pi) \Vert_{\ell^{p_1}_\pi} &= \left\Vert \Phi(x,\pi)^{-1}\sum_{k=1}^{\infty}h_k(x)\widehat{ \overline{ g_k } } (\pi)^*  \right\Vert_{\ell^{p_1}_\pi}\\
&=\left\Vert\sum_{k=1}^{\infty}h_k(x)\Phi(x,\pi)^{-1}\widehat{ \overline{ g_k } } (\pi)^*  \right\Vert_{L^{p_1}_\pi}\\
&\leq \sum_{k=1}^{\infty}\vert h_{k}(x)\vert \Vert \Phi(x,\pi)^{-1} \widehat{{ \overline{ g_k } }} (\pi)^*\Vert_{\ell^{p_1}_\pi}.
\end{align*} 
By the definition of $\ell^{p_1}_\pi$-norm, we have
\begin{align*}
    \Vert \Phi(x,\pi)^{-1}{ \overline{ g_k } } (\pi)^* \Vert_{\ell^{p_1}(\widehat{G}_0)} &=\left( \sum_{[\pi]\in \widehat{G}_0}d_\pi k_\pi^{p_1(\frac{1}{p_1}-\frac{1}{2})}\Vert \Phi(x,\pi)^{-1}   { \widehat{ \overline{ g_k } }} (\pi)^*\Vert^{p_1}_{\textnormal{HS}}\right)^{\frac{1}{p_1}}\\
    &\leq \left( \sum_{[\pi]\in \widehat{G}_0}d_\pi k_\pi^{p_1(\frac{1}{p_1}-\frac{1}{2})}\Vert \Phi(x,\pi)^{-1}\Vert_{op}\Vert  { \widehat{ \overline{ g_k } } } (\pi)^*\Vert^{p_1}_{\textnormal{HS}}\right)^{\frac{1}{p_1}}\\
    &\leq \Vert \Phi^{-1}\Vert_{\infty}\left( \sum_{[\pi]\in \widehat{G}_0}d_\pi k_\pi^{p_1(\frac{1}{p_1}-\frac{1}{2})}\Vert  { \widehat{  \overline{ g_k }} } (\pi)^*\Vert^{p_1}_{\textnormal{HS}}\right)^{\frac{1}{p_1}}\\
    &= \Vert \Phi^{-1}\Vert_{\infty}\left( \sum_{[\pi]\in \widehat{G}_0}d_\pi k_\pi^{p_1(\frac{1}{p_1}-\frac{1}{2})}\Vert  \widehat{  \overline{ g_k } } (\pi)\Vert^{p_1}_{\textnormal{HS}}\right)^{\frac{1}{p_1}}.
\end{align*}Consequently,
\begin{eqnarray}
\Vert a(x,\pi) \Vert_{\ell^{p_1}_\pi} \leq \Vert \Phi^{-1}\Vert_{\infty} \sum_{k=1}^{\infty}\vert h_{k}(x)\vert \Vert \widehat{  \overline{ g_k } } (\pi)\Vert_{\ell^{p_1}_\pi}. 
\end{eqnarray}

Now, if we use the Hausdorff-Young inequality, we deduce,  $\Vert \widehat{  \overline{ g_k } } (\pi)\Vert_{\ell^{p_1}_\pi}\leq \Vert  {g}_k\Vert_{L^{p_1'}}.$
Consequently,
\begin{align*}
    \Vert a(x,\pi)\Vert_{ \ell^{p_1}_\pi L^{p_2}_x(M\times \widehat{G}_0)}&=\left(\int\limits_{M} \left(\sum\limits_{[\pi]\in \widehat{G}_0}d_{\pi}k_{\pi}^{p_1(\frac{1}{p_1}-\frac{1}{2})}\Vert a(x,\pi)\Vert^{p_1}_{\textnormal{HS} }\right)^{\frac{p_2}{p_1}}dx\right)^{\frac{1}{p_2}}\\
    &\leq \Phi^{-1}\Vert_{\infty}\left\Vert  \sum_{k=1}^{\infty}\vert h_{k}(x)\vert \Vert \widehat{  \overline{ g_k } } (\pi)\Vert_{\ell^{p_1}_\pi} \right\Vert_{L^{p_2}_x} \\
     &\leq \sum_{k=1}^{\infty}\Vert h_{k}\Vert_{L^{p_2}}\Vert{g}_k\Vert_{L^{p_1'}}<\infty.
\end{align*}
Thus, we finish the proof.
\end{proof}

\begin{remark}
If $K={e_G}$ and $M=G$ is a compact Lie group, the condition
\begin{eqnarray}
\Vert \Phi^{-1}\Vert_{\infty}:=\sup_{(x,[\pi])\in M\times \widehat{G}_0}\Vert \Phi(x,\pi)^{-1} \Vert_{op}=\sup_{(x,[\xi])\in G\times \widehat{G}}\Vert \Phi(x,\xi)^{-1} \Vert_{op}<\infty,
\end{eqnarray} arises naturally in the context of pseudo-differential operators. Indeed, if we take $\Phi(x,\xi)=\xi(x),$ then $\Phi(x,\xi)^{-1}=\xi(x)^*,$ and 
\begin{eqnarray}
\sup_{(x,[\xi])\in G\times \widehat{G}}\Vert \xi(x)^{*} \Vert_{op}=1.
\end{eqnarray} 
\end{remark}

\begin{remark}\label{remarkcompact}
As a consequence of Delgado's theorem, if $F:L^{p}(M)\rightarrow L^{p}(M),$ $M\cong G/K,$ $1\leq p<\infty,$ is $r$-nuclear, its nuclear trace is given by
\begin{equation}
\textnormal{Tr}(F)=\int\limits_{M}\sum_{[\pi]\in \widehat{G}_0}d_\pi\textnormal{Tr}[\pi(x)^{*}\Phi(x,\pi)a(x,\pi)]dx.
\end{equation}
\end{remark}
\begin{example}[The complex projective plane $\mathbb{C}\mathbb{P}^2$] A point $\ell\in \mathbb{C}\mathbb{P}^n$ (the $n$-dimensional complex projective space) is a complex line through the origin in $\mathbb{C}^{n+1}.$ For every $n,$ $\mathbb{C}\mathbb{P}^n\cong \textnormal{SU}(n+1)/\textnormal{SU}(1)\times\textnormal{SU}(n).$  We will use the representation theory of $\textnormal{SU}(3)$ in order to describe the nuclear trace of Fourier integral operators on $\mathbb{C}\mathbb{P}^2\cong\textnormal{SU}(3)/\textnormal{SU}(1)\times\textnormal{SU}(2).$ The Lie group $\textnormal{SU}(3)$ (see \cite{Fe1}) has dimension 8 and $3$ positive square roots $\alpha,\beta$ and $\rho$ with the property
\begin{equation}
\rho=\frac{1}{2}(\alpha+\beta+\rho). 
\end{equation}
We define the weights
\begin{equation}
\sigma=\frac{2}{2}\alpha+\frac{1}{3}\beta,\,\,\,\tau=\frac{1}{3}\alpha+\frac{2}{3}\beta.
\end{equation}
With the notations above the unitary dual of $\textnormal{SU}(3)$ can be identified with
\begin{equation}
\widehat{\textnormal{SU}}(3)\cong\{\lambda:=\lambda(a,b)=a\sigma+b\tau:a,b\in\mathbb{N}_{0}, \}.
\end{equation}
In fact, every representation $\pi=\pi_{\lambda(a,b)}$ has highest weight $\lambda=\lambda(a,b)$ for some $(a,b)\in\mathbb{N}_0^2.$ In this case $d_{{\lambda(a,b)}}:=d_{\pi_{\lambda(a,b)}}=\frac{1}{2}(a+1)(b+1)(a+b+2).$ For $G=\textnormal{SU}(3)$ and $K=\textnormal{SU}(1)\times\textnormal{SU}(2),$ let us define
\begin{equation}
\mathbb{N}_{00}^2:=\{(a,b)\in\mathbb{N}_0\times \mathbb{N}_0:\pi_{\lambda(a,b)}\in \widehat{G}_0=\widehat{\textnormal{SU}}(3)_0\}.
\end{equation} Now, let us  consider a phase function $\Phi:\mathbb{C}\mathbb{P}^2\times \mathbb{N}_{00}^2\rightarrow \cup_{(a,b)\in \mathbb{N}_0\times  \mathbb{N}_0  }\textnormal{GL}(d_{\lambda(a,b)}),$  a distribution $a:\mathbb{C}\mathbb{P}^2\times \mathbb{N}_{00}^2\rightarrow \cup_{(a,b)\in \mathbb{N}\times  \mathbb{N}  }\textnormal{GL}(d_{\lambda(a,b)}), $ and the Fourier integral operator $F$ associated to $\Phi$ and to $a(\cdot,\cdot):$ 
\begin{equation}\label{cardonadefiunitionCP2}
F\varphi(\ell)=\sum_{  (a,b)\in\mathbb{N}_{00}^2 } \frac{1}{2}(a+1)(b+1)(a+b+2) \textnormal{Tr}(\Phi(\ell,(a,b))a(\ell,(a,b))\widehat{\varphi}(\pi_{\lambda(a,b)})), 
\end{equation}
where $\varphi\in C^{\infty}(\mathbb{C}\mathbb{P}^2).$ We additionally require the condition $\sigma(\ell,(a,b))_{ij}=0$ for $\ell\in \mathbb{C}\mathbb{P}^2,  $ and $i,j>k_{\pi_{\lambda(a,b)}},$ for those distributional symbols considered above. As a consequence of Remark \ref{remarkcompact}, if $F:L^{p}(\mathbb{C}\mathbb{P}^2)\rightarrow L^{p}(\mathbb{C}\mathbb{P}^2),$  $1\leq p<\infty,$ is $r$-nuclear, its nuclear trace is given by
\begin{equation}
\textnormal{Tr}(F)
=\int\limits_{\mathbb{C}\mathbb{P}^2}\sum_{ (a,b)\in\mathbb{N}_{00}^2  }(a+1)(b+1)(a+b+2) \textnormal{Tr}[\pi_{\lambda(a,b)}(\ell)^{*}\Phi(\ell,(a,b))a(\ell,(a,b))]\frac{d\ell}{2}.
\end{equation}
If $\vartheta:\mathbb{C}\mathbb{P}^2\rightarrow \textnormal{SU}(3)/\textnormal{SU}(1)\times \textnormal{SU}(2)$ is a diffeomorphism and $K=\textnormal{SU}(1)\times \textnormal{SU}(2),$ then
\begin{align*}
&\int\limits_{\mathbb{C}\mathbb{P}^2}\sum_{ (a,b)\in\mathbb{N}_{00}^2  }(a+1)(b+1)(a+b+2) \textnormal{Tr}[\pi_{\lambda(a,b)}(\ell)^{*}\Phi(\ell,(a,b))a(\ell,(a,b))]\frac{d\ell}{2}\\
&=\int\limits_{\textnormal{SU}(3)/\textnormal{SU}(1)\times \textnormal{SU}(2)}\sum_{ (a,b)\in\mathbb{N}_{00}^2  }(a+1)(b+1)(a+b+2) \\
&\hspace{5cm}\textnormal{Tr}[\pi_{\lambda(a,b)}(\vartheta^{-1}(g))^{*}\Phi(\vartheta^{-1}(g),(a,b))a(\vartheta^{-1}(g),(a,b))]\frac{dg}{2} \\
&=\int\limits_{\textnormal{SU}(3)/\textnormal{SU}(1)\times \textnormal{SU}(2)}\sum_{ (a,b)\in\mathbb{N}_{00}^2  }(a+1)(b+1)(a+b+2) \\
&\hspace{5cm}\textnormal{Tr}[\pi_{\lambda(a,b)}(g)^{*}\Phi(g,(a,b))a(g,(a,b))]\frac{dg}{2}\\
&=\int\limits_{\textnormal{SU}(3)}\sum_{ (a,b)\in\mathbb{N}_{00}^2  }(a+1)(b+1)(a+b+2) \\
&\hspace{5cm}\textnormal{Tr}[\pi_{\lambda(a,b)}(gK)^{*}\Phi(gK,(a,b))a(gK,(a,b))]\frac{d{\mu_{\textnormal{SU}(3)}}(g)}{2}
\end{align*} where we have denoted $\pi_{\lambda(a,b)}(g):=\pi_{\lambda(a,b)}(\vartheta^{-1}(g)),$ $\Phi(g,(a,b)):=\Phi(\vartheta^{-1}(g),(a,b))$ and $a(g,(a,b)):=a(\vartheta^{-1}(g),(a,b)).$ If we consider the  parametrization of $\textnormal{SU}(3)$ (see, e.g., Bronzan \cite{SU3}), $$g\equiv g(\theta_1,\theta_2,\theta_3,\phi_1,\phi_2,\phi_3,\phi_4,\phi_5):= (u_{ij})_{i,j=1,2,3}, $$ where $0\leq \theta_i\leq \frac{\pi}{2},\,0\leq \phi_i\leq 2\pi,$ and 
\begin{itemize}
\item $u_{11}=\cos\theta_1 \cos\theta_2 e^{i\phi_1}$
\item $ u_{12}=\sin\theta_1 e^{i\phi_3} $
\item $ u_{13} =\cos\theta_1\sin\theta_2 e^{i\phi_4} $
\item $u_{21} =\sin\theta_1 \sin\theta_3 e^{-i\phi_4-i\phi_5} -\sin\theta_1 \cos\theta_2 \cos\theta_3 e^{i\phi_1+i\phi_2-i\phi_3} $
\item $u_{22} =\cos\theta_1 \cos\theta_3 e^{i\phi_2}   $
\item $u_{23}=-\cos\theta_1\sin\theta_3e^{-i\phi_1-i\phi_5}-\sin\theta_1\sin\theta_2 \cos\theta_3 e^{ i\phi_2-i\phi_3+i\phi_4 }$
\item $u_{31}= -\sin\theta_1\cos\theta_2\sin\theta_3 e^{i\phi_1-i\phi_3+i\phi_5} -\sin\theta_2\cos\theta_3e^{-i\phi_2-i\phi_4}  $

\item $u_{32}=\cos\theta_1\sin\theta_3e^{i\phi_5}$
\item $u_{33} =\cos\theta_2\cos\theta_3 e^{-i\phi_1-i\phi_2}  -\sin\theta_1\sin\theta_2\sin\theta_3e^{-i\phi_3+i\phi_4+i\phi_5}, $
\end{itemize}

then, the group measure is the determinant  given by
\begin{equation}
d{\mu_{\textnormal{SU}(3)}}(g)=\frac{1}{2\pi^5}\sin\theta_1 \cos^3\theta_1\sin\theta_2\cos\theta_2\sin\theta_3 \cos\theta_3d\theta_1 d\theta_2 d\theta_3 d\phi_1 d\phi_2 d\phi_3 d\phi_4 d\phi_5,
\end{equation} and we have the trace formula \small{
\begin{align*}
{\textnormal{Tr}(F)}
=\int\limits_{0}^{\pi/2}\int\limits_{0}^{\pi/2}\int\limits_{0}^{\pi/2}\int\limits_{0}^{2\pi} \int\limits_{0}^{2\pi} \int\limits_{0}^{2\pi} \int\limits_{0}^{2\pi} \int\limits_{0}^{2\pi}\sum_{ (a,b)\in\mathbb{N}_{00}^2  }(a+b+ab+1)(a+b+2) 
\textnormal{Tr}[\pi_{\lambda(a,b)}(g(\theta,\phi)K)^{*} \\
\Phi(g(\theta,\phi)K,(a,b))  
a(g(\theta,\phi)K,(a,b))]\\
\times \frac{1}{4\pi^5}  \sin\theta_1 \cos^3\theta_1\sin\theta_2\cos\theta_2\sin\theta_3 \cos\theta_3d\theta_1 d\theta_2 d\theta_3 d\phi_1 d\phi_2 d\phi_3 d\phi_4 d\phi_5,
\end{align*}} with  $(\theta,\phi)=(\theta_1,\theta_2,\theta_3,\phi_1,\phi_2,\phi_3,\phi_4,\phi_5).$
\end{example}

\noindent {\bf{Acknowledgements:}} I would like to thanks Majid Jamalpour Birgani who provide me a preprint of the reference \cite{Majid}.

\bibliographystyle{amsplain}

\end{document}